\documentclass{amsart}
\usepackage{graphicx,amscd,amssymb,amsmath,enumerate}
\newtheorem{thm}{Theorem}[section]

\newtheorem{lem}[thm]{Lemma}
\newtheorem{prop}[thm]{Proposition}
\theoremstyle{definition}
\newtheorem{defn}[thm]{Definition}
\theoremstyle{remark}
\newtheorem{rem}[thm]{Remark}
\numberwithin{equation}{section}
\newcommand{\EE}{\mathcal{E}}

\newcommand{\OO}{\mathcal{O}}
\begin{document}

\title[Representations of Clifford Algebras]{The Fine Moduli Space of Representations of Clifford Algebras}
\author{Emre Coskun}
\address{120 Middlesex College, The University of Western Ontario, London, ON N6A 5B7 CANADA}
\email{ecoskun@uwo.ca}

\thanks{I am indebted to my thesis advisor Rajesh Kulkarni for his help on this project. I also want to thank Ajneet Dhillon for several valuable conversations involving the last section, and the referee for making valuable suggestions.}

\date{\today}
%\dedicatory{}%
%\commby{}%
% ----------------------------------------------------------------
\begin{abstract}
Given a fixed binary form $f(u,v)$ of degree $d$ over a field $k$, the associated \emph{Clifford algebra} is the $k$-algebra $C_f=k\{u,v\}/I$, where $I$ is the two-sided ideal generated by elements of the form $(\alpha u+\beta v)^{d}-f(\alpha,\beta)$ with $\alpha$ and $\beta$ arbitrary elements in $k$. All representations of $C_f$ have dimensions that are multiples of $d$, and occur in families. In this article we construct fine moduli spaces $U=U_{f,r}$ for the irreducible $rd$-dimensional representations of $C_f$ for each $r \geq 2$. Our construction starts with the projective curve $C \subset \mathbb{P}^{2}_{k}$ defined by the equation $w^d=f(u,v)$, and produces $U_{f,r}$ as a quasiprojective variety in the moduli space $\mathcal{M}(r,d_r)$ of stable vector bundles over $C$ with rank $r$ and degree $d_r=r(d+g-1)$, where $g$ denotes the genus of $C$.
\end{abstract}
\maketitle
% ----------------------------------------------------------------
\section{Introduction}
Let $f$ be a binary form of degree $d$ over a field $k$. The \emph{Clifford algebra} $C_f$ associated to $f$ is the quotient of the tensor algebra on two variables by the two-sided ideal generated by $\{(\alpha u+\beta v)^{d}-f(\alpha,\beta)\ |
\ \alpha,\beta \in k\}$. We will be interested in the case of degree $d > 3$. We also assume that $f(u, v)$ is \emph{nondegenerate}, that is, $f$ has no repeated roots over the algebraic closure of $k$.

The structure and representations of Clifford algebras have been a subject of study in many recent papers. The degree 2 case is classical; for an overview of the subject, see \cite{LM89} and \cite{iP95}. The degree 3 case was examined by Haile in \cite{dH84}. Assuming that the characteristic of the base field is not 2 or 3, and that $f(u,v)$ is nondegenerate, he proved that $C_f$ is an Azumaya algebra (see Section \ref{azumaya}) over its center. He also proved that the center is isomorphic to the coordinate ring of an affine elliptic curve $J$. The curve $J$ is the Jacobian of $w^3=f(u,v)$ and the affine elliptic curve is the complement of the identity point in $J$.

Next we describe what is known for $d>3$. Let $C$ be the curve over $k$ defined by the equation $w^{d}=f(u,v)$ in $\mathbb{P}^{2}_{k}$, and let $g$ denote its genus. Since $f$ is assumed to be nondegenerate, the curve $C$ is nonsingular, and $g=(d-1)(d-2)/2$. Haile and Tesser proved in \cite{HT88} that the dimensions of representations of $C_{f}$ are divisible by $d$. Van den Bergh proved in \cite{mvdB87} (assuming that the base field $k$ is algebraically closed of characteristic 0) that the equivalence classes of $rd$-dimensional representations of the Clifford algebra $C_{f}$ are in one-to-one correspondence with vector bundles $E$ over $C$ having rank $r$, degree $r(d+g-1)$ such that $H^{0}(E(-1))=0$. These vector bundles are always semistable and the stable bundles correspond to the irreducible representations. Then, Haile and Tesser proved in \cite{HT88} that $\widetilde{C}_f = C_f / \cap \eta$, where $\eta$ runs over the kernels of the dimension $d$ representations, is Azumaya over its center. The center of $\widetilde{C}_f$ is then the fine moduli space of $d$-dimensional representations (hence $r=1$) of $C_f$. Kulkarni then proved in \cite{rK03} that this center is the affine coordinate ring of the complement of a $\Theta$-divisor in the space $Pic^{d+g-1}_{C/k}$ of degree $d+g-1$ line bundles over $C$.

This article generalizes Kulkarni's work to the higher rank case, $r \geq 2$; under the assumption that the characteristic of $k$ does not divide $d$. We begin with Van den Bergh's correspondence between equivalence classes of $rd$-dimensional representations of $C_{f}$ and semistable vector bundles over $C$ of rank $r$ and degree $r(d+g-1)$, which can be described as follows. Consider a representation $\phi: C_{f} \to M_{rd}(k)$. Set $\alpha_{u}=\phi(u)$ and $\alpha_{v}=\phi(v)$. Then we define a map from $S=k[u,v,w]/(w^{d}-f(u,v))$ to $M_{rd}(k[u,v])$ by sending $u$ to $uI_{rd}$, $v$ to $vI_{rd}$, and $w$ to $u\alpha_{u}+v\alpha_{v}$. This makes $\bigoplus_{rd}k[u,v]$ into a graded $S$-module. It can be proven that the corresponding coherent sheaf is a vector bundle. In this way we get a rank $r$ vector bundle $\mathcal{E}$ over $C$ such that $q_{*}\mathcal{E} \cong \mathcal{O}_{\mathbb{P}^{1}_{k}}^{rd}$, where $q:C \to \mathbb{P}^{1}_{k}$ is the map defined by the inclusion $k[u,v] \to S$. The condition $H^{0}(E(-1))=0$ is equivalent to the condition that $q_{*}\mathcal{E} \cong \mathcal{O}_{\mathbb{P}^{1}_{k}}^{rd}$.

The moduli problem we solve in this article can be stated as follows: Define a contravariant functor $\mathcal{R}ep_{rd}(C_f,-)$ from the category of $k$-schemes to the category of sets by sending a $k$-scheme $S$ to the set of equivalence classes of $rd$-dimensional irreducible $S$-representations of $C_f$. (For the definition of an $S$-representation of $C_f$, see Definition \ref{s-rep}.) Procesi proved that this functor is representable by a scheme $U$ (see Theorem 1.8 in Chapter 4 of \cite{cP73}.) Unfortunately, Procesi's method gives no geometric description of $U$. In this article, we prove that $U$ is isomorphic to an open subset of the coarse moduli space $\mathcal{M}(r,r(d+g-1))$ of stable vector bundles of rank $r$ and degree $r(d+g-1)$ over $C$. Using this geometric description, we construct the universal representation $\mathcal{A}$ of $C_f$ over $U$.

More explicitly, the universal representation of $C_f$ of a given dimension $rd$ is a $k$-algebra homomorphism $\psi:C_f \to H^0(\mathcal{A})$, where $\mathcal{A}$ is a sheaf of Azumaya algebras of rank $(rd)^2$ defined over $U$. The base variety $U$ is the open subset of $\mathcal{M}(r,r(d+g-1))$ consisting of stable vector bundles $E$ such that $H^{0}(E(-1))=0$. The sheaf $\mathcal{A}$ is constructed as follows. There is a Quot scheme $Q$ (see Theorem \ref{groth}), that parametrizes the quotients of the trivial vector bundle of large enough rank $N$ over $C$, having rank $r$ and degree $r(d+g-1)$, and there is a universal bundle $\mathcal{E}$ over $C \times Q$. We take the open subset $\Omega$ of $Q$ consisting of stable vector bundles $E$ with $H^{0}(E(-1))=0$. We prove in Lemma \ref{pf} that the pushforward of $\mathcal{E}$ to $\Omega$ under the projection map $\pi: C \times \Omega \to \Omega$ is a rank $rd$ vector bundle. The algebraic group $GL(N)$ acts on $\Omega$ and also on $\pi_{*} \mathcal{E}$. The stabilizer of a point in $\Omega$ under this action is the group of scalar matrices, so the action of $GL(N)$ on $\Omega$ descends to an action of $PGL(N)$. But the scalar matrices act as scalar multiplication on $\pi_{*} \mathcal{E}$. So we get a $PGL(N)$-action on $\mathcal{E}nd(\pi_{*}\mathcal{E})$. The resulting Geometric Invariant Theory quotient is the variety $U$ together with a sheaf of algebras $\mathcal{A}$ on it. We then construct the homomorphism $\psi:C_f \to H^0(\mathcal{A})$.

The main theorem of this article is as follows:
\begin{thm}\textbf{(Main Theorem)}
Let $k$ be an algebraically closed base field. $\mathcal{R}ep_{rd}(C_f,-)$ is represented by the pair $(\psi,\mathcal{A})$ described above.
\end{thm}

For the proof, we construct an isomorphism between $(\psi,\mathcal{A})$ and the universal representation in Procesi's theorem. As mentioned above, Procesi showed that the functor $\mathcal{R}ep_{rd}(C_f,-)$ is representable, that is, there is a universal representation $(\Psi,\mathcal{B})$ consisting of a scheme $T$, a sheaf of Azumaya algebras $\mathcal{B}$ over $T$, and a $k$-algebra homomorphism $\Psi:C_f \to H^0(\mathcal{B})$. Since we have an irreducible representation $(\psi,\mathcal{A})$ over $U$, we obtain a map $\alpha: U \to T$ such that $\alpha^{*}(\Psi,\mathcal{B}) \cong (\psi,\mathcal{A})$. The idea of the proof is to construct an inverse to $\alpha$. To do this, we first consider irreducible representations of the type $(\phi,\mathcal{E}nd(\mathcal{E}^{S}))$, where $\mathcal{E}^{S}$ is a vector bundle of rank $rd$ over $S$. We construct an associated vector bundle in Lemma \ref{const_M}. By the coarse moduli property of $\mathcal{M}(r,r(d+g-1))$, this gives us a map $f:S \to U$. We use this to construct a morphism $\beta: T \to U$. We then prove that $\alpha$ and $\beta$ are inverses, and hence that $(\psi,\mathcal{A})$ represents $\mathcal{R}ep_{rd}(C_f,-)$. This finishes the proof of the theorem.

Let $B$ be a $k$-algebra. As mentioned above, the functor $\mathcal{R}ep_{n}(B,-)$ that assigns the set of irreducible $S$-representations of $B$ to the $k$-scheme $S$ is representable by Procesi's theorem. However, explicit descriptions of the scheme $Rep_n$ that represents $\mathcal{R}ep_n$ are rare. The current article is of interest in this direction because it provides an explicit description of $Rep_n$ for the Clifford algebra as an open subvariety of the quasiprojective variety that is the coarse moduli space of stable, rank $r$ and degree $r(d+g-1)$ vector bundles over the curve $C$.

Further questions can be asked about this universal representation. $\mathcal{A}$ gives a class in the Brauer group $Br(U)$. Since the Brauer group is torsion, it is natural to ask what the period and index of this class (as defined in Section \ref{azumaya}) are. It is known that the period always divides the index, and the set of primes dividing both of them is the same. Hence the index divides a power of the period. The \emph{period-index} problem is the problem of computing this power. This is part of an ongoing project.

Second, it is an interesting question to examine the representations of Clifford algebras of \emph{ternary} forms. By the results of Van den Bergh in \cite{mvdB87}, these correspond to vector bundles over a surface $X$ in $\mathbb{P}^3$ defined by the equation $z^d=f(u,v,w)$, whose direct images under the natural projection map $X \to \mathbb{P}^{2}$ is a trivial vector bundle. These will be studied in future articles.

\subsection*{Conventions and notation}
Let $k$ be a perfect, infinite base field with characteristic 0 or not dividing $d$. A \emph{variety} means a separated scheme of finite type over $k$. A variety of dimension 1 is called a \emph{curve}.

\begin{itemize}
\item All rings have an identity element.
\item All schemes are locally Noetherian over $k$ and all morphisms are locally of finite type over $k$.
\item The terms line bundle and invertible sheaf are used interchangeably.
\item $\mathcal{M}(r,d)$ denotes the coarse moduli space of stable vector bundles of rank $r$ and degree $d$ over a curve $C$.
\item For any vector bundle $\mathcal{E}$, $\chi(\mathcal{E})$ denotes the Euler characteristic of $\mathcal{E}$.
\item The projection of a fiber product onto the $i^{\mbox{th}}$ component $X_{i}$ is denoted $p_{i}$, $\pi_{i}$, $p_{X_{i}}$ or $\pi_{X_{i}}$. When no subscript is indicated, $\pi$ is the canonical map from $C \times Y \to Y$ for a variety $Y$.
\item For technical reasons, we assume throughout the article that the binary form $f(x,y)$ has no repeated factors over the algebraic closure of the base field $k$ and that the characteristic of $k$ does not divide $d$.
\item $q$ denotes the canonical map from $C$ to $\mathbb{P}^{1}$, where $C$ is the curve $w^d=f(u,v)$ in $\mathbb{P}^{2}$. $q_{S}$ denotes the map $q \times id_{S}: C \times S \to \mathbb{P}^{1} \times S$ for any scheme $S$, and $p_{S}: \mathbb{P}^{1} \times S \to S$ is the natural projection.
\item For any coherent sheaf $\mathcal{F}$ over a scheme $X$, we denote the $i^{th}$ cohomology of $\mathcal{F}$ by $H^{i}(\mathcal{F})$ whenever $X$ is clear from the context.
\item For a closed point $y$ in a scheme $S$, and for a vector bundle $\mathcal{E}$ over $C \times S$, $\mathcal{E}_{y}$ denotes the pull-back of $\mathcal{E}$ under the canonical map $id \times i_y: C \times Spec\ k(y) \to C \times S$.
\end{itemize}

The following lemma will be useful in proving the main theorem.
\begin{lem}\label{easy-lemma}
Let $Y$ be a reduced quasi-projective variety over an algebraically closed field $k$. Let $f:Y \to Y$ be a morphism of $k$-varieties that is the identity on closed points. Then $f$ is the identity morphism.
\end{lem}
\begin{proof}
First we prove that $f$ is the identity on all points. Let $\xi$ be a non-closed point. Assume that $f(\xi)=\eta \neq \xi$. Then since closed irreducible subsets have unique generic points in a scheme (see 2.1.2 and 2.1.3 Chapitre 0 and Corollaire 1.1.8 Chapitre 1, \cite{aG65} Vol. 1 for more details), we have $\overline{\xi} \neq \overline{\eta}$, where $\overline{\xi}$ and $\overline{\eta}$ denote the closures of $\xi$ and $\eta$, respectively. There are two cases to consider:

Case 1: $\overline{\xi} \nsubseteq \overline{\eta}$. In this case, $\overline{\xi} \setminus \overline{\eta}$ is a nonempty open subset $V$ of $\overline{\xi}$. Since $V$ is a quasi-projective variety over $k$, it has a closed point. Pick a closed point $y \in \overline{\xi}\ \backslash\ \overline{\eta}$. Then $f^{-1}(\overline{\eta})$ is a closed set that contains $\xi$, and hence it contains $\overline{\xi}$ and in particular $y$, a contradiction.

Case 2: $\overline{\xi} \subsetneq \overline{\eta}$. Since this is a proper inclusion, the dimension of $\overline{\eta}$ is strictly greater than the dimension of $\overline{\xi}$. But since $k(\eta)$ is a subfield of $k(\xi)$, this is a contradiction for dimension reasons.

Since $f$ is the identity map on points, it maps affine open subsets to affine open subsets.

Now since any reduced quasi-projective variety can be covered by reduced open affine varieties, without loss of generality we may assume that $Y$ affine, and hence it is the spectrum of a finitely generated reduced $k$-algebra $A=A(Y)=k[T_1,\ldots,T_n]/I(Y)$. Let $f$ be induced by the morphism of rings $\phi:A \to A$.

Let $a \in A$. We may view $a$ as a morphism $a:Y \to \mathbb{A}^{1}$. Then the composition $a \circ f: Y \to Y \to \mathbb{A}^{1}$ corresponds to $\phi(a) \in A$. Since $f$ is the identity on points; as morphisms $Y \to \mathbb{A}^{1}$, $a$ and $\phi(a)$ take the same value on all the points on $Y$. This means that, for any prime ideal $\mathfrak{p}$ of $A$, we have $a-\phi(a) \in \mathfrak{p}$. Since $A$ is reduced, this means that $a-\phi(a)=0$ and hence $\phi$ is the identity morphism.

%Recall that $f$ is the identity on points. Now since $\phi$ induces the identity morphism on the corresponding varieties, that $\phi$ is the identity follows from the classical equivalence between the category of affine algebraic varieties over $k$ and the category of finitely generated reduced $k$-algebras.

%From elementary algebraic geometry, the image $\phi(T_i)$ is the regular function on $Y$ whose values are equal to that of the $i^{th}$ coordinate function on $Y$. Hence, the element $\phi(T_i)-T_i$ of $A(Y)$ vanishes on $Y$. Since $Y$ is reduced, and hence Since $Y$ is reduced, this is the same as $T_i$. This proves that $\phi$, and hence $f$, is the identity map.
\end{proof}

\section{Preliminaries}
In this section, we will review the results to be used in this article.
\subsection{Moduli of vector bundles over curves}\label{mb}
Here we review how to construct the moduli spaces of vector bundles over curves. (For more information, see \cite{peN78}, \cite{jL97} and \cite{gF95}.) Let $C$ be a nonsingular irreducible projective curve of genus $g \geq 2$ and let $E$ be a vector bundle over $C$. The degree $deg(E)$ is defined to be the degree of the determinant line bundle $det(E)$ of $E$, and can be any integer. A \emph{family} of vector bundles over $C$ parametrized by a scheme $S$ is a vector bundle over $C \times S$ that is flat over $S$; and an \emph{isomorphism} of families is just an isomorphism of vector bundles over $C \times S$. For a family $E$ of vector bundles parametrized by $S$ and $s \in S$, we denote by $E_s$ the fiber of $E$ over $s$.

To be able to define the moduli space of vector bundles over $C$ with rank $r \geq 2$ and degree $D$ as a variety, we have to introduce extra conditions on the vector bundles. To do that, we define the \emph{slope} of a vector bundle $E$ over $C$ as $\mu(E)=deg(E)/rk(E)$. Then we have the following:
\begin{defn}
A vector bundle $E$ over $C$ is \emph{stable} (\emph{semistable}) if, for every nontrivial subbundle $F$ with $F \neq E$,
\[
 \mu(F) < \mu(E)\quad(\leq)
\]
\end{defn}
With these definitions in place, we can now state the main results about the moduli space of vector bundles over $C$:
\begin{thm}
There exist coarse moduli spaces $\mathcal{M}(r,D)$ and $\mathcal{M}^{ss}(r,D)$ for stable and semistable bundles of rank $r$ and degree $D$ over any nonsingular irreducible projective curve $C$ of genus $g \geq 2$. $\mathcal{M}(r,D)$ is a nonsingular quasiprojective variety that is contained in $\mathcal{M}^{ss}(r,D)$, which is a projective variety. $\mathcal{M}^{ss}(r,D)$ is normal, and its singular locus is given by $\mathcal{M}^{ss}(r,D)\backslash \mathcal{M}(r,D)$. The dimension of these moduli spaces is equal to $r^{2}(g-1)+1$. Moreover, $\mathcal{M}(r,D)$ is a fine moduli space if and only if $gcd(r,D)=1$.
\end{thm}
The points of $\mathcal{M}(r,D)$ correspond to stable vector bundles of rank $r$ and degree $D$ over $C$. To describe the points of $\mathcal{M}^{ss}(r,D)\backslash \mathcal{M}(r,D)$, we note that for any semistable bundle $E$ over $C$, there is a sequence of subbundles
\[
 E_1 \subseteq E_2 \subseteq \ldots E_n = E
\]
such that $E_1$, $E_2 / E_1$, \ldots, $E_n / E_{n-1}$ are all stable with slopes equal to $\mu(E)$. Moreover, it can be shown that the bundle $gr\ E = E_1 \oplus E_2 / E_1 \oplus \ldots \oplus E_n / E_{n-1}$ is determined up to isomorphism by $E$. (This is called the Jordan-H\"{o}lder filtration.) We have
\begin{prop}
Two semistable bundles $E$ and $E'$ determine the same point of $\mathcal{M}^{ss}(r,D)$ if and only if $gr\ E \cong gr\ E'$.
\end{prop}

Let us now recall how the coarse moduli space $\mathcal{M}(r,D)$ of stable vector
bundles of rank $r$ and degree $D$ is constructed. The following lemma is crucial:

\begin{lem}\label{lem21}
Let $E$ be a semistable vector bundle over $C$ of rank $r$ and degree $D$,
and suppose that $D > r(2g-1)$. Then
\begin{enumerate}
\item $H^{1}(E)=0$
\item $E$ is generated by its sections.
\end{enumerate}
\end{lem}

For large enough $m$, the degree of $E(m)=E \otimes \mathcal{O}_{C}(m)$ is greater than $r(2g-1)$ and the lemma allows us to write it as a quotient of a trivial vector bundle $\mathbb{E}=\bigoplus_{N} \mathcal{O}_{C}$ over $C$. To determine the necessary rank $N$ of this trivial vector bundle, which
is the same as $h^{0}(E(m))$, we can simply use the Riemann-Roch
theorem:
\begin{thm}
\textbf{(Riemann-Roch)} Let $E$ be a vector bundle over a curve $C$ with genus $g \geq 2$. Then we have:
\[
 \chi(E)=h^0(E)-h^1(E)=(1-g)rk(E)+deg(E).
\]
\end{thm}
Let $r=rk(E)$, and $d=deg(\mathcal{O}_{C}(1))$ where $\mathcal{O}_{C}(1)$ is a fixed very ample line bundle on $C$. Note that using the Riemann-Roch theorem, we have
\[
 \chi(E(m))=(1-g)r+deg(E(m))=(1-g)rk(E)+deg(E)+rdm.
\]
Hence specifying the rank and degree of a vector bundle determines its Hilbert polynomial $P(m)$.

Let $\mathbb{E}=\bigoplus_{N} \mathcal{O}_{C}$ be as above, and let $P$ be a linear polynomial with integer coefficients. We denote by $Q=Q(\mathbb{E},P)$ the family of all coherent sheaves $\mathcal{F}$ on $C$ together with a surjection $\mathbb{E} \to \mathcal{F}$ such that the Hilbert polynomial of $\mathcal{F}$ is $P$. Now, the main tool in the construction is the following theorem of Grothendieck. (See Th\'{e}or\`{e}me 3.1, \cite{aG61}.) The statement below is from Theorem 6.1, \cite{cS67}.
\begin{thm}\label{groth}
\textbf{(Grothendieck)} There is a unique projective algebraic variety structure on $Q=Q(\mathbb{E},P)$, and a surjection $\theta:p_{1}^{*}(\mathbb{E}) \to \mathcal{E}$ of coherent sheaves on $C \times Q$, where $p_{1}$ is the
canonical projection $C \times Q \to C$, such that:

1: $\mathcal{E}$ is flat over $Q$;

2: the restriction of the homomorphism $\theta:p_{1}^{*}(\mathbb{E}) \to \mathcal{E}$ to $C \times q \cong C$, $q \in Q$; when viewed as a surjection $\mathbb{E} \to \mathcal{E}|_{C \times q}$, corresponds to the element of $Q(\mathbb{E},P)$ represented by $q$;

3: given a surjection $\phi:p_{1}^{*}(\mathbb{E}) \to G$ of coherent sheaves on $C \times T$, where $T$ is an algebraic scheme such that $G$ is flat over $T$, and the Hilbert polynomial of the restriction of $G$ to $C \times t \cong C$ is $P$, there exists a unique morphism $f:T \to Q$ such that $\phi:p_{1}^{*}(\mathbb{E})
\to G$ is the inverse image of $\theta:p_{1}^{*}(\mathbb{E}) \to \mathcal{E}$ by the morphism $f$.
\end{thm}

Note that the group $GL(N)$ may be identified with the group of automorphisms of $\mathbb{E}$, hence $GL(N)$ acts on $Q$, and also on the sheaf $\mathcal{E}$. The action of $GL(N)$ on $Q$ goes down into an action of $PGL(N)$, but it does not go into an action of $PGL(N)$ on $\mathcal{E}$. The scalar multiples of identity act as scalar multiplication on $\mathcal{E}$. (This is the reason why the moduli space for vector bundles over curves is not fine in general.)

Let $R^s$ be the subset of $Q$ consisting of those $x$ in $Q$ for which the bundle $\mathcal{E}_{x}$ is stable. This is an open subset of $Q$ on which $PGL(N)$ acts freely and hence has a quotient. This quotient is the moduli space $\mathcal{M}(r,D')$.

Now we discuss some properties of $\mathcal{E}$. We will assume that the rank $r$ and the degree $D$ are given, and that the vector bundles can be twisted by $\mathcal{O}_{C}(m)$ to make their degrees $D'$ larger than $r(2g-1)$, as required in \ref{lem21}. By Theorem 5.3, \cite{peN78}, the bundle $\mathcal{E}$ has the \emph{local universal property} for families of bundles of rank $r$ and degree $D'$ which satisfy conditions (1) and (2) in \ref{lem21}: Let $\mathcal{F}$ be a family of vector bundles over $C$ with rank $r$, degree $D'$ and satisfying (1) and (2); parametrized by a scheme $T$ and flat over $T$. Then we can cover $T$ with open subsets $T_i$ and we can find maps $f_i: T_i \to Q$ such that $\mathcal{F}$ is isomorphic to $(id_C \times f_i)^* \mathcal{E}$. We do \emph{not} require the maps $f_i$ to be unique.

Let $U$ denote the subset of $\mathcal{M}(r,r(d+g-1))$ consisting of vector bundles $E$ over $C$ such that $H^{0}(E(-1))=0$. We prove that this is a nonempty open subset. The fact that it is nonempty was proven in \cite{mvdB87}, Theorem 2.4. To prove that it is open, we look at the subset $\Omega$ of $R^s$ consisting of vector bundles with the same property. This is a smooth subset. (For details, see \cite{peN78} and \cite{jL97}.) Then $PGL(N)$ acts freely on $\Omega$ and we can take the GIT-quotient to construct $U$.

\begin{lem}
$\Omega$ is open in $R^s$.
\end{lem}
\begin{proof}
Note that the second projection $C \times R^s \to R^s$ is a projective morphism.
Since any affine subset of $R^s$ is noetherian, we can restrict to an affine open
subset after choosing an affine open cover. We also note
that $\mathcal{E} \otimes p_C^* \mathcal{O}_{C}(-m-1)$ is a coherent sheaf on $C \times R^s$
and is flat over $R^s$, by Grothendieck's theorem. That the set
$\Omega$ is open in $R^s$ now follows from the
Semicontinuity Theorem, Thm. 12.8, \cite{rH77}.
\end{proof}

\subsection{The Clifford algebra and its representations}
Let $f(u,v)$ be a binary form of degree $d$ over $k$. We define the \emph{Clifford algebra} of $f$, denoted $C_f$ to be the associative $k$-algebra $k\{u,v\}/I$, where $I$ is the two-sided ideal generated by elements of the form $(\alpha u+\beta v)^d-f(\alpha,\beta)$, where $\alpha$ and $\beta$ are arbitrary elements of $k$. A \emph{representation} of $C_f$ is a $k$-algebra homomorphism $\phi:C_f \to M_m(F)$, where $F$ is a field extension of $k$. The integer $m$ is called the \emph{dimension} of the representation.

Let $C$ be the curve in $\mathbb{P}^2$ defined by the equation $w^d=f(u,v)$, where $u$, $v$ and $w$ are the projective coordinates. We assume that the binary form $f(u,v)$ does not have any repeated factors over an algebraic closure of $k$ and that the characteristic of $k$ does not divide $d$. With these assumptions, we have the following lemma:
\begin{lem}
$C$ is a smooth curve of degree $d$.
\end{lem}
\begin{proof}
It is obvious that $C$ is a curve of degree $d$. To prove that it is smooth, consider the partial derivatives of the defining equation $w^d-f(u,v)$:

\begin{align*}
 \partial_w (w^d-f(u,v))&=dw^{d-1} \\
 \partial_u (w^d-f(u,v))&=-\partial_u f(u,v) \\
 \partial_v (w^d-f(u,v))&=-\partial_v f(u,v)
\end{align*}

It is now obvious that for a point $[u:v:w] \in C$ to be singular, $f(u,v)$ and both its partial derivatives have to vanish on it. But since we assumed that $f(u,v)$ has no repeated factors, this is not possible.
\end{proof}
From now on, $C$ will denote this curve. We note that the genus of $C$ is $g=(d-1)(d-2)/2$. Assuming that $d \geq 4$, we have $g \geq 2$. We also note that the map $[u:v:w] \mapsto [u:v]$ defines a degree $d$ map $p: C \to \mathbb{P}^1$.

We now prove that the rank of a representation of $C_f$ is divisible by $d$:
\begin{prop}
(\cite{HT88}, Proposition 1.1) Let $f$ be a binary form of degree $d$ over an infinite field $k$ with no repeated factors over an algebraic closure of $k$. If $\phi$ is a representation of the Clifford algebra $C_f$, then the degree $d$ of $f$ divides the rank of $\phi$.
\end{prop}

We now want to describe representations of $C_f$ in more detail. Let $\phi:C_f \to M_m(k)$ be a representation. Let $R=k[u,v]$ have the standard grading and let $S=k[u,v,w]/(w^d-f(u,v))$. Note that $X=\emph{Proj }S$, and the map $p$ is induced by the inclusion $R \to S$. Let $\alpha_u=\phi(u)$ and $\alpha_v=\phi(v)$. These two matrices define a map of graded algebras $\phi_f: S \to M_m(R)$ by sending $u$ to $u I_m$, $v$ to $v I_m$ and $w$ to $u \alpha_u + v \alpha_v$. Conversely, if we have such a map, we can define a representation of $f$ by taking $\phi_f(u I_m)$ and $\phi_f(v I_m)$. Via this map, $R^m$ becomes a graded $S$-module. In this way, we get a vector bundle $E$ over $X$ such that $p_* E$ is trivial of rank $m$. This vector bundle $E$ also satisfies $H^0(E(-1))=0$.

\subsection{Azumaya algebras}\label{azumaya}
We follow the discussions in \cite{dS99} and \cite{jsM80} for this review. Let $A$ be an algebra over a commutative ring $R$. We assume that $R$ is the center of $A$. Then $A$ is called \emph{Azumaya} over $R$ if $A$ is faithful, finitely generated and projective as an $R$-module and the map $\phi_{A}: A \otimes_{R} A^{\circ} \to End_{R}(A)$ defined by
\[
 \phi_{A}(\sum_i r_i \otimes s_i)(r)=\sum_i r_i r s_i
\]
is an isomorphism.

The following proposition will be useful later:
\begin{prop}
Let $A$ and $B$ be Azumaya algebras over $R$. Then $A \otimes_R B$ is Azumaya over $R$.
\end{prop}

By a construction similar to obtaining a quasicoherent sheaf over $Spec\ R$ using an $R$-module $M$, given an Azumaya algebra $A$ over $R$, we can obtain a sheaf of algebras over $Spec\ R$. We can use this as the motivation for the definition:
\begin{defn}\label{defn-azumaya}
Let $X$ be a scheme. An $\mathcal{O}_{X}$-algebra $A$ is called an \emph{Azumaya algebra} over $X$ if it is coherent as an $\mathcal{O}_{X}$-module and if, for all closed points $x$ of $X$, $A_{x}$ is an Azumaya algebra over the local ring $\mathcal{O}_{X,x}$.
\end{defn}
The conditions in Definition \ref{defn-azumaya} imply that $A$ is locally free of finite rank as an $\mathcal{O}_{X}$-module.

Instead of using the definition to prove that an $\mathcal{O}_X{}$-algebra is Azumaya, we will make use of the following proposition:
\begin{prop}\label{flat}
Let $A$ be an $\mathcal{O}_X{}$-algebra that is of finite type as an $\mathcal{O}_{X}$-module. Then $A$ is an Azumaya algebra over $X$ if and only if there is a flat covering $(U_i \to X)$ of $X$ such that for each $i$, $A \otimes_{\mathcal{O}_X{}} \mathcal{O}_{U_i} \cong M_{r_i}(\mathcal{O}_{U_i})$ for some $r_i$.
\end{prop}

\begin{defn}\label{s-rep}
If $S$ is a $k$-scheme then an $S$-representation of dimension $n$ of $B$ is a pair $(\phi, \mathcal{O}_A)$, where $\mathcal{O}_A$ is a sheaf of Azumaya algebras of rank $n^2$ over $S$ and $\phi:B \to H^0(S,\mathcal{O}_A)$ is a ring homomorphism. Two representations $(\phi_1, \mathcal{O}_{A_1})$ and $(\phi_2, \mathcal{O}_{A_2})$ are called \emph{equivalent} if there is an isomorphism $\theta: \mathcal{O}_{A_1} \to \mathcal{O}_{A_2}$ of sheaves of rings such that $\phi_2=H^0(S,\theta) \circ \phi_1$.  A representation of $B$ is called \emph{irreducible} if the image of $B$ generates $\mathcal{O}_A$ locally. Let $\mathcal{R}ep_{n}(B,S)$ be the set of equivalence classes of irreducible $S$-representations of degree $n$ of $B$. This defines a contravariant functor $\mathcal{R}ep_{n}(B,-)$ from the category of $k$-schemes to the category of sets.
\end{defn}

\begin{thm}
(Theorem 4.1, \cite{mvdB89}) The functor $\mathcal{R}ep_{n}(B,-)$ is representable in $(Sch/k)$.
\end{thm}

Sometimes it is more convenient and easier to work with representation into endomorphism sheaves of vector bundles. Let $\mathcal{G}_n(B,-)$ be the subfunctor of $\mathcal{R}ep_{n}(B,S)$ consisting of representations of endomorphism sheaves of vector bundles of rank $n$. This is not a sheaf with respect to the flat topology. However, its sheafification with respect to the flat topology is well-known:
\begin{lem}
(Lemma 4.2, \cite{mvdB89}) $\mathcal{R}ep_{n}(B,-) \cong \overline{\mathcal{G}}_n(B,-)$, where $\overline{\mathcal{G}}_n(B,-)$ denotes the sheafification of $\mathcal{G}_n(B,-)$.
\end{lem}

We end this section with a lemma that will be useful later. Let $S$ be a scheme, and $\mathcal{O}$ be a sheaf of $\mathcal{O}_{S}$-algebras over $S$. A collection of global sections $(\sigma_i \in H^0(S,\mathcal{O}))_{i \in I}$ is said to generate $\mathcal{O}$ if the stalks $(\sigma_i)_s$ generate $\mathcal{O}_s$ for all $s \in S$.
\begin{lem}\label{sheaves_of_algebras}
Let $S$ and $\mathcal{O}$ be as above. If the $\sigma_i(s)$ generate the \emph{fibers} $\mathcal{O}(s)$ for \emph{closed} points $s \in S$, then the $\sigma_i$ generate $\mathcal{O}$.
\end{lem}
\begin{proof}
The statement is local in $S$, so we will assume that $S=Spec\ R$ for a ring $R$ and that $\mathcal{O}=A^\sim$ for an $R$-algebra $A$.

First, we claim that if the $\sigma_i(s)$ generate $\mathcal{O}_s$ for closed points $s$ as a $k(s)$-algebra, then the $(\sigma_i)_s$ generate $\mathcal{O}_s$ as an $R_s$-algebra. But this follows from Nakayama's lemma.

Second, we claim that if the $(\sigma_i)_s$ generate $\mathcal{O}_s$ for closed $s$, then they generate $\mathcal{O}_s$ for all $s \in S$. The $\sigma_i$ correspond to elements $a_i \in A$. We know that for all maximal ideals $\mathfrak{m}$ in $R$, $(a_i)_{\mathfrak{m}}$ generate $A_{\mathfrak{m}}$ as an $R_{\mathfrak{m}}$-algebra.

Let $B$ be the $R$-subalgebra of $A$ generated by the $a_i$. Then $B_{\mathfrak{m}}=A_{\mathfrak{m}}$. By Corollary 2.9, \cite{dE99}, $A=B$ and the lemma is proved.
\end{proof}

\section{Construction of the Universal Representation}\label{s:con}
In this section, we construct a sheaf of Azumaya algebras over the variety $U$ defined in the previous section. Recall that there is a coarse moduli space $\mathcal{M}(r,D)$ of vector bundles over the curve $C$ with rank $r$ and degree $D$; this is a quasiprojective variety. The variety $U$ is the open subset of $\mathcal{M}(r,D)$ consisting of stable vector bundles $E$ over $C$ such that $H^0(E(-1))=0$.

We first recall how to construct quotients of vector bundles. Let $Y$ be an integral algebraic variety and $G$ an algebraic group acting on $Y$. We have the following definition:
\begin{defn}
(Definition 8.4.3, \cite{jL97}.) An algebraic vector $G$-bundle $F \to
Y$ is a vector bundle over $Y$, equipped with a $G$-action which is
linear in each fiber and such that the diagram
\[
    \begin{CD}
        G \times F @>>> F\\
        @VVV            @VVV\\
        G \times Y @>>> Y
    \end{CD}
\]
commutes. In other words, for every $y \in Y$, we have a linear map $F_y \to F_{g.y}$

$F$ is said to \emph{descend} to $M$ if there is a vector bundle
$F'$ over $M$ such that the algebraic vector $G$-bundles $F$ and $\pi^{*}F'$ are
isomorphic.
\end{defn}
This definition can also be stated in terms of sheaves. Let $\mathcal{F}$ denote the sheaf of sections of the vector bundle $F$. For every $g \in G$ and every open subset $V$ of $Y$, we have a linear map $\mathcal{F}(V) \to \mathcal{F}(g.V)$. These maps are required to satisfy the obvious compatibility conditions.

The following lemma gives a necessary and sufficient condition for an algebraic vector $G$-bundle $F$ to descend to $M$:
(For the proof, see Theorem 2.3, \cite{DN89}.)
\begin{lem}
Let $G$, $Y$ and $M$ be as before. Let $F \to Y$ be an algebraic
vector $G$-bundle over $Y$. Then $F$ descends to $M$ if and only if for
every closed point $y$ of $Y$ such that the orbit of $y$ is closed, the
stabilizer of $G$ at $y$ acts trivially on $F_{y}$.
\end{lem}

Recall that $\Omega$ is the subset of the Quot-scheme $Q$ consisting of vector bundles $E$ over the curve $C$ such that $H^0(E(-1))=0$. We also have the bundle $\mathcal{E}$ parametrizing quotients of the trivial vector bundle $\mathbb{E}$ having the given Hilbert polynomial $P$. Recall also that $GL(N)$ acts on $\Omega$, and the stabilizers of points are the scalar matrices. Hence, there is an induced action of $PGL(N)$ on $\Omega$, and the good quotient is the variety $U$.

When we try to take the quotient of $\mathcal{E}$ by $PGL(N)$, however; we are unable to define an action of $PGL(N)$ on $\mathcal{E}$ because of the fact that the scalar multiples of identity in $GL(N)$ do not act trivially.

We resolve this difficulty as follows. Consider the direct image of $\mathcal{E}$ under the projection $\pi: C \times \Omega \to \Omega$. Then consider the endomorphism bundle $\mathcal{E}nd(\pi_{*}\mathcal{E})$ and then define an action of $PGL(N)$ on it. (We prove that $\pi_{*}\mathcal{E}$ is a vector bundle below in Lemma \ref{pf}.) Since $GL(N)$ only acts on the second component of $C \times \Omega$, this gives a $GL(N)$-action on $\pi_{*}\mathcal{E}$. Now let $GL(N)$ act on $\mathcal{E}nd(\pi_{*}\mathcal{E})$ by conjugation, see below that the action of $GL(N)$ descends to a $PGL(N)$-action: The action of the stabilizer of a point on the fibers of $\pi_{*}\mathcal{E}$ is by scalar multiplication, but on $\mathcal{E}nd(\pi_{*}\mathcal{E})$, this action becomes trivial. So the action of $GL(N)$ on $\mathcal{E}nd(\pi_{*}\mathcal{E})$ descends to $PGL(N)$.

To be precise, let $g \in GL(N)$, $f \in \Gamma(V,\mathcal{E}nd(\pi_{*}\mathcal{E}))$ for an open subset $V \subset \Omega$. Then $f$ is an endomorphism of $\pi_{*}\mathcal{E}$ over $V$, and we define $g.f$ to be an endomorphism of $\pi_{*}\mathcal{E}$ over $g.V$ as follows. Let $s$ be a section of $\pi_{*}\mathcal{E}$ over an open subset $W \subset V$. Then $g.f$ is the endomorphism of $\pi_{*}\mathcal{E}$ that sends $s$ to $g(f(g^{-1}s))$. It is obvious that this defines a $GL(N)$-action on $\mathcal{E}nd(\pi_{*}\mathcal{E})$. Since a scalar matrix $\lambda I$ acts as multiplication by $\lambda$ on $(\pi_{*}\mathcal{E})_{x}$, it can easily be seen as acting trivially on $\mathcal{E}nd(\pi_{*}\mathcal{E})_{x}$:
\begin{align*}\label{pgln-action}
 (\lambda I \cdot f)(s) &= (\lambda I \circ f \circ (\lambda  I)^{-1})(s) \\
 &= \lambda I \circ f (\frac{1}{\lambda}s) \\
 &= \lambda I \circ \frac{1}{\lambda} f (s) \\
 &= \lambda \frac{1}{\lambda} f (s) \\
 &= f (s).
\end{align*}
This gives us a $PGL(N)$ action on $\mathcal{E}nd(\pi_{*}\mathcal{E})$. As seen above, we
have a quotient vector bundle $\mathcal{E}nd(\pi_{*}\mathcal{E})^{PGL(N)}$, which we will denote as $\mathcal{A}$ over $U$.

Recall that there is a canonical map $q:C \to \mathbb{P}^{1}$, corresponding to
the inclusion $k[u,v] \to k[u,v,w]/(w^{d}-f(u,v))$. Recall that for a closed point $y$ in $\Omega$ $\mathcal{E}_{y}$ denotes the
pull-back of $\mathcal{E}$ under the canonical map $id \times
i_{y}:C \times Spec\ k(y) \to C \times \Omega$.

We prove the following proposition, which follows from a standard
result of Grothendieck.
\begin{prop}\label{pushforward}
The coherent sheaf $q_{*}(\mathcal{E}_{y})$ is isomorphic to the trivial bundle $\bigoplus_{i=1}^{rd} \mathcal{O}_{\mathbb{P}^{1}_{k(y)}}$.
\end{prop}
\begin{proof}
We consider the case of an algebraically closed base field $k$ first. Since $\mathbb{P}^{1}_{k}$ is a nonsingular projective curve, any torsion-free coherent sheaf is a vector bundle. Since $q_{*}$ respects torsion-freeness, $q_{*}(\mathcal{E}_{y})$ is a vector bundle on $\mathbb{P}^{1}_{k}$. It has rank $rd$ since $\mathcal{E}_{y}$ has rank $r$ and the map $q$ has degree $d$. Since any vector bundle on $\mathbb{P}^{1}_{k}$ is a sum of line bundles, we can write $q_{*}(\mathcal{E}_{y}) \cong \bigoplus^{rd}_{i=1} \mathcal{O}_{\mathbb{P}^{1}_{k}}(n_{i})$ for some integers $n_{i}$. We have $\chi(q_{*} \mathcal{E}_{y}))=\chi(\mathcal{E}_{y})$. Then,
\[
 \chi(q_{*}(\mathcal{E}_{y}))=(\mbox{rk}(q_{*}(\mathcal{E}_{y})))(1-g_{\mathbb{P}^{1}_{k}})+\mbox{deg}(q_{*}(\mathcal{E}_{y}))\\
                            = rd+\sum_{i=1}^{rd}n_{i},
\]
\[
 \chi(\mathcal{E}_{y})=(\mbox{rk}(\mathcal{E}_{y}))(1-g)+\mbox{deg}(\mathcal{E}_{y})=r(1-g)+r(d+g-1)=rd.
\]
So we have $\sum_{i=1}^{rd} n_{i}=0$.

We also have, by the projection formula:
\[
 h^{0}(\mathbb{P}^{1},q_{*}(\mathcal{E}_{y}\otimes_{\mathcal{O}_{C}}
 \mathcal{O}_{C}(-1)))=h^{0}(\mathbb{P}^{1},(q_{*}\mathcal{E}_{y})(-1))=h^{0}(C,\mathcal{E}_{y}(-1))=0.
\]
So it follows that $h^{0}(\mathbb{P}^{1},(q_{*}\mathcal{E}_{y})(-1))=0$. This is equal to $h^{0}(\mathbb{P}^{1},\bigoplus_{i=1}^{rd}\mathcal{O}_{\mathbb{P}^{1}}(n_{i}-1))$. It then follows that $n_{i}-1 < 0$ for all $i$ and hence $n_{i}=0$ because the sum of the $n_{i}$ is equal to 0.

Now assume that $k$ is an arbitrary field, and let $\overline{k}$ denote its algebraic closure. Consider the pullback of $\EE_y$ along the canonical morphism $i_C:C \times \overline{k} \to C \times k(y)$. Then we have $(q_{\overline{k}})_* (i_C^*(\EE_y)) \cong \bigoplus_{rd} \OO_{\mathbb{P}^{1}}$. By Proposition 9.3, Chapter 3, \cite{rH77}; we have $(q_{\overline{k}})_* (i_C^*(\EE_y)) \cong (i_{\mathbb{P}^{1}}^*)(q_{k(y)})_*\EE_y$. But we have $Aut((i_{\mathbb{P}^{1}}^*)(q_{k(y)})_*\EE_y)=GL_{rd}(\overline{k})$ and $H^1(Gal_{\overline{k}/k(y)},GL_{rd}(\overline{k}))=0$ by Hilbert 90. So we have $q_{*}(\mathcal{E}_{y}) \cong \bigoplus_{i=1}^{rd} \mathcal{O}_{\mathbb{P}^{1}_{k(y)}}$.
\end{proof}
We can use this result to prove the following result, which is due to Kulkarni. (See Proposition 3.5, \cite{rK03}.)
\begin{lem}\label{pf}
Let $\pi: C \times \Omega \to \Omega$ denote the projection onto the
second factor. Then $\pi_{*}\mathcal{E}$ is a locally free sheaf of
rank $rd$.
\end{lem}
\begin{proof}
For any point $y$ in $\Omega$, let $k(y)$ be the residue field of
$y$. Denote by $\mathcal{E}_{y}$ the vector bundle $(id \times
i_{y})^{*} \mathcal{E}$ on $C_{k(y)}$, where $i_{y}$ is the
inclusion of the point $y$; that is, $Spec \
k(y) \to \Omega$. We have to prove that
$\mbox{dim}_{k(y)}H^{0}(C_{k(y)},\mathcal{E}_{y})$ is constant, and
equal to $rd$. Recall that $\Omega$ is irreducible, and reduced. So, by Corollary 2
pg. 50, \cite{dM70}, it will follow that $\pi_{*}\mathcal{E}$ is a
locally free sheaf of rank $rd$.

If $y$ is closed, then $q_{*}\mathcal{E}_{y} \cong
\bigoplus_{i=1}^{rd} \mathcal{O}_{\mathbb{P}^{1}}$ by Proposition \ref{pushforward}. So
$h^{0}(C,\mathcal{E}_{y})=h^{0}(\mathbb{P}^{1},q_{*}\mathcal{E}_{y})=rd$.
It is also well-known that the function $y \mapsto \mbox{dim}_{k(y)}
H^{0}(C_{k(y)},\mathcal{E}_{y})$ is upper semicontinuous. (See, for
example, Theorem 12.8 in Chapter 3, \cite{rH77}.) Together with the
fact that $\Omega$ is a Jacobson scheme (since it is locally of finite type over the spectrum of a field), it follows that
$\mbox{dim}_{k(y)} H^{0}(C_{k(y)},\mathcal{E}_{y})$ is constant and
equal to $rd$ for all $y \in \Omega$. This implies that
$\pi_{*}\mathcal{E}$ is a vector bundle of rank $rd$.
\end{proof}

\begin{lem}\label{com}
Let $S$ be a scheme, and let $f: S \to \Omega$
be a morphism. Consider the commutative diagram:
\[
\begin{CD}
 C \times S @>\mbox{id}_{C} \times f>> C \times \Omega \\
 @Vq_{S}VV @VVq_{\Omega}V \\
 \mathbb{P}^{1} \times S @>\mbox{id}_{\mathbb{P}^{1}} \times f>> \mathbb{P}^{1} \times \Omega.
\end{CD}
\]
Then the coherent sheaves $(id_{\mathbb{P}^{1}} \times f)^{*}(q_{\Omega})_{*}\mathcal{E}$ and
$(q_{S})_{*}(id_{C} \times f)^{*}\mathcal{E}$ are isomorphic.
\end{lem}
\begin{proof}
Since $\mathcal{E}$ is a coherent sheaf on $C \times \Omega=Proj(\mathcal{O}_{\Omega}[u,v,w]/(w^d-f))$, there exists a sheaf of graded $\mathcal{O}_{\Omega}[u,v,w]/(w^d-f)$-modules $M$ such that $\mathcal{E}$ is isomorphic to $\widetilde{M}$. (See \cite{aG65}, II, Sections 3.2 and 3.3.) Then we have the following isomorphisms, as sheaves of $\mathcal{O}_{\Omega}[u,v]$-modules:
\[
 (id_{\mathbb{P}^{1}} \times f)^{*}(q_{\Omega})_{*}\mathcal{E} \cong (\mathcal{O}_{S}[u,v] \otimes_{\mathcal{O}_{\Omega}[u,v]}
 {M}_{\mathcal{O}_{\Omega}[u,v]})^{\sim},
\]
and
\[
 (q_{k(y)})_{*}(id_{C} \times f)^{*}\mathcal{E} \cong {(\mathcal{O}_{S}[u,v,w]/(w^d-f(u,v)) \otimes_{\mathcal{O}_{\Omega}[u,v,w]/(w^d-f(u,v))} M)_{\mathcal{O}_{S}[u,v]}^{\sim}}.
\]
But the two graded $k(y)[u,v]$-modules on the right side of the equations above are isomorphic. So the lemma is proved.
\end{proof}

For the remainder of this section, we construct a map $\psi:C_{f} \to
H^{0}(\mathcal{A})$ that is a universal representation for $C_{f}$.
The next theorem is crucial in this construction. It is
proven in \cite{rH77} that for any morphism $g: X \to Y$ of schemes
and a sheaf $\mathcal{G}$ of $\mathcal{O}_{X}$-modules, there is a
natural morphism $g^{*}g_{*}\mathcal{G} \to \mathcal{G}$. This is a
consequence of the adjointness of the functors $g^{*}$ and $g_{*}$.
The proof of the following theorem follows the discussion in
\cite{rK03} closely. (See, Proposition 3.8, \cite{rK03}.)
\begin{thm}\label{usls}
Let $\mathcal{F}=(q_{\Omega})_{*}\mathcal{E}$. Then the natural morphism $u: p_{\Omega}^{*}
(p_{\Omega})_{*} \mathcal{F} \to \mathcal{F}$ is an isomorphism. (Recall that $p_{\Omega}$ is the canonical
projection $\mathbb{P}^{1} \times \Omega \to \Omega$.)
\end{thm}
\begin{proof}
We know that $(p_{\Omega})_{*}\mathcal{F}=\pi_{*}\mathcal{E}$ is a locally free sheaf of
rank $rd$ by Lemma \ref{pf}. So the sheaf $p_{\Omega}^{*}(p_{\Omega})_{*}\mathcal{F}$
is also a locally free sheaf of rank $rd$ on
$\mathbb{P}^{1}_{\Omega}$.

First we claim that $\mathcal{F}$ is a locally free sheaf of rank
$rd$ on $\mathbb{P}^{1}_{\Omega}$. We prove that $dim_{k(y)}
\mathcal{F} \otimes k(y)$ is $rd$ for any closed point $y$ in
$\mathbb{P}^{1}_{\Omega}$. This will be sufficient by the upper semicontinuity of the
dimension function and by the fact that $\mathbb{P}^{1}_{\Omega}$ is locally
of finite type over $k$.

For any closed point $y$ in $\mathbb{P}^{1} \times \Omega$, consider
the following commutative diagram:

\[
    \begin{CD}
        \mbox{Spec }k(y)            @=              \mbox{Spec }k(y)\\
        @ViVV                                        @VVjV\\
        \mathbb{P}^{1}\times Spec(k(y))       @>\mbox{id}\times i_{y}>>    \mathbb{P}^{1} \times \Omega\\
        @Vp_{k(y)}VV                                 @VVp_{\Omega}V\\
        \mbox{Spec }k(y)            @>i_{y}>>    \Omega
    \end{CD}
\]

We prove that $\mbox{dim}_{k(y)} \mathcal{F}
\otimes_{\mathcal{O}_{\mathbb{P}^{1} \times \Omega}} k(y) =
\mbox{dim}_{k(y)} j^{*}\mathcal{F}$ is $rd$. But
$\mbox{dim}_{k(y)}j^{*}\mathcal{F}=\mbox{dim}_{k(y)}i^{*}(id \times
i_{y})^{*} \mathcal{F}$. From Lemma \ref{com} and Proposition \ref{pushforward}, it follows that
$(id \times i_{y})^{*} \mathcal{F} \cong (id \times i_{y})^{*} (q_{\Omega})_{*}\mathcal{E}
\cong q_{*}(id \times i)^{*}\mathcal{E} \cong q_{*}(\mathcal{E}_{y})$ is a trivial
vector bundle of rank $rd$, so that $\mbox{dim}_{k(y)}i^{*}(id \times i_{y})^{*}
\mathcal{F}$ is $rd$.

Now $u$ is a morphism of vector bundles of rank $rd$. To prove that
it is an isomorphism, it is enough to show that $u_{x}:
(p_{\Omega}^{*} (p_{\Omega})_{*} \mathcal{F})_{x} \to
(\mathcal{F})_{x}$ is bijective for all points $x$ in
$\mathbb{P}^{1} \times \Omega$. But it is sufficient to prove this
only for closed points (\cite{aG65}, I, Corollary 0.5.5.7).

By (\cite{aG65}, I, Corollary 0.5.5.6) it is sufficient to prove that
\[
 u_{y} \otimes id: (p_{\Omega}^{*}(p_{\Omega})_{*}\mathcal{F})_{y} /
 \mathfrak{m}_{y}(p_{\Omega}^{*}(p_{\Omega})_{*}\mathcal{F})_{y} \to
 \mathcal{F}_{y}/\mathfrak{m}_{y}\mathcal{F}_{y}
\]
is surjective for all closed points $y$ in
$\mathbb{P}^{1}_{\Omega}$. But this homomorphism is surjective if
and only if the morphism
\[
 j^{*}u: j^{*}p_{\Omega}^{*}(p_{\Omega})_{*}\mathcal{F} \to
 j^{*}\mathcal{F}
\]
is surjective, which is the same as
\[
 i^{*}(id \times i_{y})^{*}u:i^{*}(id \times
 i_{y})^{*}p_{\Omega}^{*}(p_{\Omega})_{*}\mathcal{F} \to i^{*}(id \times
 i_{y})^{*} \mathcal{F}
\]
being surjective. So it is sufficient to prove that
\[
 (id \times i_{y})^{*}u: (id \times i_{y})^{*}
 p_{\Omega}^{*}(p_{\Omega})_{*}\mathcal{F} \to (id \times i_{y})^{*}
 \mathcal{F}
\]
is an isomorphism. But we have the isomorphism
\[
 (id \times i_{y})^{*}p_{\Omega}^{*}(p_{\Omega})_{*}\mathcal{F}
 \cong p_{k(y)}^{*}i_{y}^{*}(p_{\Omega})_{*}\mathcal{F}.
\]
Note that $i_{y}^{*}(p_{\Omega})_{*}\mathcal{F}$ is a trivial vector
bundle of rank $rd$, and so
\[
 p_{k(y)}^{*}i_{y}^{*}(p_{\Omega})_{*}\mathcal{F} \cong p_{k(y)}^{*}
 (\bigoplus_{rd}\mathcal{O}_{Spec k(y)}) \cong \bigoplus_{rd}
 \mathcal{O}_{\mathbb{P}^{1}_{k(y)}}.
\]
This proves that $(id \times
i_{y})^{*}p_{\Omega}^{*}(p_{\Omega})_{*} \mathcal{F}$ is a trivial
vector bundle. Recall that $(id \times i_y)^{*}\mathcal{F}$ is a trivial vector bundle of rank $rd$. So the morphism $(id \times i_{y})^{*}u$ will be an
isomorphism if it is so on the global sections. But the morphism
\[
 p_{k(y)}^{*}i_{y}^{*}(p_{\Omega})_{*}\mathcal{F} \to (id \times
 i_{y})^{*}\mathcal{F}
\]
on the global sections is an isomorphism if the natural morphism
\[
 (p_{\Omega})_{*}\mathcal{F} \otimes_{\mathcal{O}_{\Omega}} k(y) \to
 H^{0}(\mathbb{P}^{1}_{k(y)},\mathcal{F}_{y})
\]
is an isomorphism. By the earlier part, and Corollary 2, p. 50, \cite{dM70}, this is the
case.
\end{proof}

Recall that we have the morphisms $q_{\Omega}:C \times \Omega=C_{\Omega} \to \mathbb{P}^{1} \times \Omega={\mathbb{P}^{1}}_{\Omega}$ and $p_{\Omega}:\mathbb{P}^{1}_{\Omega} \to \Omega$. Consider $\mathcal{O}_{\Omega}[u,v]$ and $\mathcal{O}_{\Omega}[u,v,w]/(w^d-f)$. These are sheaves of graded $\mathcal{O}_{\Omega}$-algebras generated by degree 1 elements. We can define their homogeneous spectra $Proj(\mathcal{O}_{\Omega}[u,v])=\mathbb{P}^{1}_{\Omega}$ and $Proj(\mathcal{O}_{\Omega}[u,v,w]/(w^d-f))=C \times \Omega$. Recall also that we have the vector bundle $\mathcal{F}$ on $\mathbb{P}^{1}_{\Omega}$. That this is a vector bundle follows from Lemma \ref{pf} and Theorem \ref{usls}.

Let $\mathcal{G}$ be a sheaf of $\mathcal{O}_{\mathbb{P}^{1}_{\Omega}}$-modules. Define
\[
 \Gamma_{*}(\mathcal{G})=\bigoplus_{n \in \mathbb{Z}}(p_{\Omega})_{*}(\mathcal{G}(n)).
\]
In particular, we have
\[
 \Gamma_{*}(\mathcal{O}_{\mathbb{P}^{1}_{\Omega}})=\bigoplus_{n \in \mathbb{Z}}(p_{\Omega})_{*}(\mathcal{O}_{\mathbb{P}^{1}_{\Omega}}(n)).
\]
$\Gamma_{*}(\mathcal{O}_{\mathbb{P}^{1}_{\Omega}})$ is then a sheaf of graded $\mathcal{O}_{\mathbb{P}^{1}_{\Omega}}$-algebras and $\Gamma_{*}(\mathcal{G})$ becomes a sheaf of graded $\Gamma_{*}(\mathcal{O}_{\mathbb{P}^{1}_{\Omega}})$-modules.

\begin{rem}
These are particular cases of constructions carried out in \cite{aG65}, II, 3.3.
\end{rem}

By Proposition 7.11, Chapter 2, \cite{rH77}, we have
\[
 \Gamma_{*}(\mathcal{O}_{\mathbb{P}^{1}_{\Omega}}) \cong \mathcal{O}_{\Omega}[u,v]
\]
as sheaves of graded $\mathcal{O}_{\Omega}$-modules.

In the next proposition, we use Theorem \ref{usls} to describe $\Gamma_{*}(\mathcal{F})$.
\begin{prop}\label{pm}
The $\mathcal{O}_{\Omega}[u,v]$-module $\Gamma_{*}(\mathcal{F})=\bigoplus_{n \in \mathbb{Z}} (p_{\Omega})_{*}(\mathcal{F}(n))$ is isomorphic to $(\pi_{*}\mathcal{E}) \otimes_{\mathcal{O}_{\Omega}} \mathcal{O}_{\Omega}[u,v]$ as a sheaf of graded $\mathcal{O}_{\Omega}[u,v]$-modules.
\end{prop}
\begin{proof}
By the theorem, it is enough to compute the graded module associated to the coherent sheaf $p_{\Omega}^{*}(p_{\Omega})_{*}\mathcal{F}$. By projection formula, we have
\[
 (p_{\Omega})_{*}(\mathcal{F}(n))=(p_{\Omega})_{*}(\mathcal{F}) \otimes_{\mathcal{O}_{\Omega}} (p_{\Omega})_{*} (\mathcal{O}_{\mathbb{P}^{1}_{\Omega}}(n))
\]
By definition:
\begin{align*}
 \Gamma_{*}(\mathcal{F}) &= \bigoplus_{n \in \mathbb{Z}} (p_{\Omega})_{*}(\mathcal{F}(n)) \\
 &=\bigoplus_{n \in \mathbb{Z}} ((p_{\Omega})_{*}(\mathcal{F}) \otimes_{\mathcal{O}_{\Omega}} (p_{\Omega})_{*} (\mathcal{O}_{\mathbb{P}^{1}_{\Omega}}(n))) \\
 &\cong (p_{\Omega})_{*}(\mathcal{F}) \otimes_{\mathcal{O}_{\Omega}} (\bigoplus_{n \in \mathbb{Z}}(p_{\Omega})_{*} (\mathcal{O}_{\mathbb{P}^{1}_{\Omega}}(n))) \\
 &= (p_{\Omega})_{*}(\mathcal{F}) \otimes_{\mathcal{O}_{\Omega}} \Gamma_{*}(\mathcal{O}_{\mathbb{P}^{1}_{\Omega}}) \\
 &\cong (\pi_{*}\mathcal{E}) \otimes_{\mathcal{O}_{\Omega}} \mathcal{O}_{\Omega}[u,v]
\end{align*}
as desired.
\end{proof}

Next we consider the question of the existence of the universal representation. This representation will be an algebra homomorphism $\psi: C_{f} \to H^{0}(U,\mathcal{A})$ that satisfies a universal property that will be discussed in the next section. Recall that $\mathcal{A}$ is the Azumaya algebra obtained by taking the quotient of $\mathcal{E}nd(\pi_{*}\mathcal{E})$ by the action of $PGL(N)$.

\begin{thm}\label{psi}
There exists an algebra homomorphism $\psi: C_{f} \to H^{0}(U,\mathcal{A})$.
\end{thm}
\begin{proof}
We prove this theorem by showing the existence of elements in $H^{0}(U,\mathcal{A})$ that satisfy the relations of the Clifford algebra.

First we make a simple observation. Let $\pi_{*}\mathcal{E}$ be as above, and let $\mathcal{M}$ be the graded $\mathcal{O}_{\Omega}[u,v]$-module $\pi_{*}\mathcal{E} \otimes_{\mathcal{O}_{\Omega}} \mathcal{O}_{\Omega}[u,v]$. Note that $\mathcal{M}_{0}=\pi_{*}\mathcal{E}$. $\pi_{*}\mathcal{E}$ is an $\mathcal{O}_{\Omega}$-module and the grading is determined by assigning $u$ and $v$ as degree 1 elements. Let $\mathcal{M}_{i}$ be the $i^{th}$ graded piece of $\mathcal{M}$. Then we have
\[
 Hom_{\mathcal{O}_{\Omega}}(\mathcal{M}_{0},\mathcal{M}_{1})=u End_{\mathcal{O}_{\Omega}}(\mathcal{M}_{0})+v End_{\mathcal{O}_{\Omega}}(\mathcal{M}_{0}).
\]
This follows from the fact that $\mathcal{M}_{1}=u\mathcal{M}_{0} \oplus v\mathcal{M}_{0}$.

Consider the vector bundle $\mathcal{E}$ on $C_{\Omega}$. We have
\begin{align*}
 \Gamma_{*}(\mathcal{E}) &= \bigoplus_{n \in \mathbb{Z}} \pi_{*}(\mathcal{E}(n)) \\
 &= \bigoplus_{n \in \mathbb{Z}} (p_{\Omega})_{*}(q_{\Omega})_{*}(\mathcal{E}(n)) \\
 &= \bigoplus_{n \in \mathbb{Z}} (p_{\Omega})_{*}((q_{\Omega})_{*}(\mathcal{E}) \otimes_{\mathcal{O}_{\mathbb{P}^{1}_{\Omega}}} \mathcal{O}_{\mathbb{P}^{1}_{\Omega}}(n)) \\
 &= \bigoplus_{n \in \mathbb{Z}} (p_{\Omega})_{*}(\mathcal{F}(n)) \\
 &= \Gamma_{*}(\mathcal{F})
\end{align*}
Hence, $\Gamma_{*}(\mathcal{F})$, which we proved to be isomorphic to $(\pi_{*}\mathcal{E}) \otimes_{\mathcal{O}_{\Omega}} \mathcal{O}_{\Omega}[u,v]$, can also be viewed as a sheaf of graded modules over $\Gamma_{*}(\mathcal{O}_{C_{\Omega}}) \cong \mathcal{O}_{\Omega}[u,v,w]/(w^d-f)$. Now $w$ is a homogeneous element of degree one in this graded $\mathcal{O}_{\Omega}$-algebra, so we can view $w$ (to be precise, the multiplication map by $w$) as an element of $Hom_{\mathcal{O}_{\Omega}} (\mathcal{M}_{0},\mathcal{M}_{1})$. By the comment above, there exist elements $\alpha_{u}$ and $\alpha_{v}$ in $End_{\mathcal{O}_{\Omega}}(M_{0})$ such that we have the following equality in $Hom_{\mathcal{O}_{\Omega}} (\mathcal{M}_{0},\mathcal{M}_{1})$:
\[
 w=u \alpha_{u}+v \alpha_{v}.
\]
Now we can consider the element $w^{d}$ as an element of $Hom_{\mathcal{O}_{\Omega}} (M_{0},M_{d})$. The relation $w^{d}=f(u,v)$ holds in $Hom_{\mathcal{O}_{\Omega}} (\mathcal{M}_{0},\mathcal{M}_{d})$ as well. This shows that the elements $\alpha_{u}$ and $\alpha_{v}$ in $End_{\mathcal{O}_{\Omega}}(M_{0})$ satisfy the relations of the Clifford algebra $C_{f}$. So we get a homomorphism:
\[
 \chi:C_{f} \to
 End_{\mathcal{O}_{\Omega}}(\mathcal{M}_{0})
\]

Finally, recall that $PGL(N)$ acts on $End_{\mathcal{O}_{\Omega}}(\mathcal{M}_0)$, as was shown in the beginning of the section. Since $w$ is invariant under this action, it follows that $\alpha_{u}$ and $\alpha_{v}$ are also $PGL(N)$-invariant. Hence they give global sections of $\mathcal{A}$ that satisfy the relationships of the Clifford algebra; and that means that we have a map $\psi: C_{f} \to H^0(\mathcal{A})$.
\end{proof}

Now we prove the following proposition, which will be used in the proof of the main theorem.
\begin{prop}\label{important_remark}
Suppose that $k$ is algebraically closed. Let $x \in U$ be a closed point, and $i_x: Spec\ k \to U$ the inclusion. The pullback $i_x^{*} (\psi,\mathcal{A})$ is an $rd$-dimensional representation of $C_f$ that corresponds to the vector bundle $E$ over $C$ defined by the point $x$ under Van den Bergh's correspondence. (See Lemma 2, \cite{mvdB87}.)
\end{prop}
\begin{proof}
Note that any closed point $x \in U$ can be lifted to a closed point $y \in \Omega$. We have the following diagram:
\[
\begin{CD}
 Spec\ k(y) @= Spec\ k(x) \\
 @Vi_{y}VV     @VVi_{x}V \\
 \Omega @>>\mathfrak{q}> U
\end{CD}
\]
Recall that $\mathfrak{q}:\Omega \to U$ is the good quotient map by the action of $PGL(N)$. Hence, we have
\begin{align*}
 i_{x}^{*}\mathcal{A} &= (\mathfrak{q} \circ i_y)^{*}\mathcal{A} \\
 &\cong i_{y}^{*}\mathfrak{q}^{*}\mathcal{A} \\
 &\cong i_{y}^{*}\mathcal{E}nd(\pi_{*}\mathcal{E}) \\
 &\cong \mathcal{E}nd(i_{y}^{*}\pi_{*}\mathcal{E})=\mathcal{E}nd((\pi_{*}\mathcal{E})_{y})
\end{align*}
Since $\mathcal{E}nd(\pi_{*}\mathcal{E})$ is an endomorphism bundle of dimension $(rd)^2$, $i_x^{*}\mathcal{A}$ also has dimension $(rd)^2$. Note also that $i_x^{*}(\alpha_u)$ and $i_x^{*}(\alpha_v)$ satisfy the relations of the Clifford algebra. This proves that $i_x^{*}\mathcal{A}$ is an $rd$-dimensional representation of $C_f$.

We claim that this is the same representation that corresponds to the point $x$ under Van den Bergh's correspondence. Recall the construction of a representation of $C_{f}$ from a stable vector bundle $E$ over $C$ with rank $r$, degree $r(d+g-1)$, and $H^{0}(E(-1))=0$: Consider the direct image $q_{*}(E)$. This is a trivial vector bundle of rank $rd$ on $\mathbb{P}^{1}$ by Proposition \ref{pushforward}. Its associated graded module over $k[u,v]$ is $\bigoplus_{rd} k[u,v]$, and this is also a graded module over $k[u,v,w]/(w^d-f(u,v))$. The action of $w$ gives two matrices in $M_{rd}(k)$ satisfying the relations of the Clifford algebra and hence a map of algebras $C_f \to M_{rd}(k)$. Conversely, let $\phi$ be a representation of $C_f$. Consider the two matrices $\phi(u)$ and $\phi(v)$. The associated graded module of the trivial vector bundle of rank $rd$ on $\mathbb{P}^1$ over $k[u,v]$ is $\bigoplus_{rd} k[u,v]$. We can define an action of $w$ given by the images of the generators of the Clifford algebra, i.e. $w$ acts as $u \phi(u)+v \phi(v)$ on $\bigoplus_{rd} k[u,v]$. This makes $\bigoplus_{rd} k[u,v]$ into a graded module over $k[u,v,w]/(w^d-f)$. In this way, we get a stable rank $r$ vector bundle $E$ on the curve $C$, such that the degree of $E$ is $r(d+g-1)$ and $H^0(E(-1))=0$. (See Section 1, \cite{mvdB87}.)

Recall that the two sections $\alpha_u$ and $\alpha_v$ are defined by the action of $w$ on the $\mathcal{O}_{\Omega}[u,v,w]/(w^d-f)$-module $\Gamma_{*}(\mathcal{F})=\pi_{*}\mathcal{E} \otimes_{\mathcal{O}_{\Omega}} \mathcal{O}_{\Omega}[u,v]$. Therefore, the two sections $i_y^*(\alpha_u)$ and $i_y^*(\alpha_v)$ are determined by the action of $w$ on the $k(y)[u,v,w]/(w^d-f)$-module $\Gamma_{*}(\mathcal{F})\otimes_{\mathcal{O}_{\Omega}} k(y)$.

Now we have
\begin{align*}
 \Gamma_{*}(\mathcal{F})\otimes_{\mathcal{O}_{\Omega}} k(y) &\cong (\pi_{*}\mathcal{E} \otimes_{\mathcal{O}_{\Omega}} \mathcal{O}_{\Omega}[u,v]) \otimes_{\mathcal{O}_{\Omega}} k(y) \\
 &\cong (\pi_{*}\mathcal{E})_{y} \otimes_{k(y)} k(y)[u,v]
\end{align*}
as graded $k(y)[u,v,w]/(w^d-f)$-modules and $(\pi_{*}\mathcal{E})_{y}$, which is the restriction of $\pi_{*}\mathcal{E}$ to the closed point $y \in \Omega$, is a vector space of dimension $rd$. So we have $\Gamma_{*}(\mathcal{F})\otimes_{\mathcal{O}_{\Omega}} k(y) \cong \bigoplus_{rd}k(y)[u,v]$.

Recall that any vector bundle corresponding to a closed point $y \in \Omega$ is such that its direct image under $q:C \to \mathbb{P}^1$ is trivial of rank $rd$. So the isomorphism class of $\mathcal{E}_y$ is determined by giving an action of $w$ on the $k(y)[u,v]$-module $\bigoplus_{rd}k(y)[u,v]$ such that $w^d=f(u,v)$. Since the two $w$-actions agree, the corresponding vector bundles are isomorphic.
\end{proof}

We finish this section by proving that the $U$-representation $(\psi,\mathcal{A})$ is an irreducible $C_f$-representation as defined in Section \ref{azumaya}.
\begin{prop}\label{irred}
The pair $(\psi, \mathcal{A})$ is an irreducible
$U$-representation of dimension $rd$ of the Clifford algebra $C_{f}$.
\end{prop}
\begin{proof}
First we have to prove that $\mathcal{A}$ is an Azumaya algebra.
Using Corollary 8.3.6 of \cite{jL97}, we can cover $U$ with \'{e}tale maps $\rho_{i}:V_{i} \to U$ and $PGL(N)$-equivariant maps $\tau_{i}: V_{i} \times PGL(N) \to \Omega$ such that the diagrams
\[
    \begin{CD}
        V_{i} \times PGL(N)         @>\tau_{i}>>     \Omega\\
        @Vpr_{1}VV                                   @VVqV\\
        V_{i}                       @>\rho_{i}>>     U
    \end{CD}
\]
are cartesian.

Now it is obvious that the composition $\rho_{i} \circ pr_{1}$ is flat. The pullback of $\mathcal{A}$ along this map is isomorphic to the pullback of $\mathcal{A}$ along $q \circ \tau_{i}$. But we have:
\begin{align*}
  (q \circ \tau_{i})^{*} \mathcal{A} &\cong \tau_{i}^{*} q^{*} \mathcal{E}nd(\pi_{*}\mathcal{E})^{PGL(N)} \\
                                     &\cong \tau_{i}^{*} \mathcal{E}nd(\pi_{*}\mathcal{E}) \\
                                     &\cong \mathcal{E}nd(\tau_{i}^{*} \pi_{*}\mathcal{E})
\end{align*}
It now follows from Proposition \ref{flat} that $\mathcal{A}$ is an Azumaya algebra.
The fact that the dimension is $rd$ follows from the fact that the pullback of $\mathcal{A}$ along the quotient map $\Omega \to U$ is a rank $rd$ vector bundle. Finally, the map $\psi: C_f \to H^0(\mathcal{A})$ was constructed in Theorem \ref{psi} and irreducibility follows from Proposition \ref{important_remark} and Lemma \ref{sheaves_of_algebras}.
\end{proof}

\section{The Moduli Problem}\label{s:moduli}
In this section, we assume $k$ to be algebraically closed.

As stated in the introduction, Procesi proved (see Theorem 1.8, Chapter 4, \cite{cP73}) that the functor $\mathcal{R}ep_{rd}(C_f,-)$ is representable. In this section, we prove the main theorem; which states that $\mathcal{R}ep_{rd}(C_f,-)$ is represented by $U$ and the $U$-representation $(\psi,\mathcal{A})$ as defined in the previous section. Recall that $U$ is defined to be the open subset of the moduli space $\mathcal{M}(r,r(d+g-1))$ consisting of stable vector bundles $E$ over $C$ such that $H^0(E(-1))=0$. These correspond to the irreducible representations of $C_f$. See \cite{mvdB87}. We constructed a sheaf of Azumaya algebras $\mathcal{A}$ on $U$ by considering the direct image $\pi_{*}\mathcal{E}$ of the vector bundle $\mathcal{E}$ on $C \times \Omega$ under the projection map $\pi:C \times \Omega \to \Omega$ and then considering the action of $PGL(N)$ on $\mathcal{E}nd(\pi_{*}\mathcal{E})$. Taking the quotient gives us $\mathcal{A}$. In Theorem \ref{psi}, we constructed an algebra homomorphism $\psi:C_f \to H^0(\mathcal{A})$ and in Proposition \ref{irred} we proved that this makes the pair $(\psi, \mathcal{A})$ into a $C_f$-representation. In this section, we will prove that this representation is the universal representation for the functor $\mathcal{R}ep_{rd}(C_f,-)$.

If $S$ is a $k$-scheme, then by an \emph{$S$-representation of dimension $n$}
of $C_{f}$ we mean a pair $(\psi,\mathcal{O}_{A})$, where
$\mathcal{O}_{A}$ is a sheaf of Azumaya algebras of dimension $n^{2}$
over $S$ and $\psi:C_{f} \to H^{0}(S,\mathcal{O}_{A})$ is a
$k$-algebra homomorphism. Two $S$-representations
$(\psi_{1},\mathcal{O}_{A_{1}})$ and
$(\psi_{2},\mathcal{O}_{A_{2}})$ are called equivalent if there is
an isomorphism $\theta: \mathcal{O}_{A_{1}} \to \mathcal{O}_{A_{2}}$
of sheaves of Azumaya algebras such that $\psi_{2}=H^{0}(S,\theta)
\circ \psi_{1}$.

We call an $S$-representation of $C_{f}$ is called \emph{irreducible} if the image of $C_{f}$ generates $\mathcal{O}_{A}$ locally.

Let $\mathcal{R}ep_{n}(C_{f},-)$ be the functor that assigns to a
$k$-scheme $S$ the set of equivalence classes of irreducible
$S$-representations of degree $n$ of $C_{f}$. Since Azumaya algebras
pull back to Azumaya algebras and irreducible representations are
stable under pull-back, it follows that $\mathcal{R}ep_{n}(C_{f},-)$
is indeed a functor.

It is known that this functor is representable in $Sch/k$. See, for example, Theorem 4.1, \cite{mvdB89}. Our goal in this section is to identify the scheme which represents it. Recall that we have the open subset $U$ of $\mathcal{M}(r,r(d+g-1))$ before and the sheaf of Azumaya algebras $\mathcal{A}=\mathcal{E}nd(\pi_{*}\mathcal{E})^{PGL(N)}$ on it. We will prove that this pair represents the functor $\mathcal{R}ep_{rd}(C_{f},-)$; and here we use the morphism $\psi$ defined and proven to exist in the previous section. (See Theorem \ref{psi} and Proposition \ref{irred}.)

We also consider representations of $C_{f}$ into endomorphism sheaves of vector bundles. Let $\mathcal{G}_{rd}(C_{f},-)$ be the subfunctor of $\mathcal{R}ep_{rd}(C_{f},-)$ that assigns to a $k$-scheme $S$ the set of equivalence classes of irreducible $S$-representations into endomorphism sheaves of vector bundles of rank $rd$. Again, since endomorphism sheaves of vector bundles pull back to sheaves of the same kind, it follows that this is also a functor. It is not a sheaf (with respect to the fppf topology); however, it can be proven that its sheafification $\overline{\mathcal{G}_{rd}}$ is isomorphic to $\mathcal{R}ep_{rd}$. (See Lemma 4.2, \cite{mvdB89}). This fact will be useful to us later when we prove the main theorem.

Let $\eta : C_{f} \to M_{rd}(K)$ be a representation of $C_f$, where $K$ is a field extension of $k$. Denote the images of the two generators of $C_{f}$ under
$\eta$ by $\overline{\alpha_{u}}$ and $\overline{\alpha_{v}}$.
Consider the morphism:

\begin{align*}
 S_{K}=\frac{K[u,v,w]}{(w^{d}-f(u,v))} &\to M_{rd}(K[u,v])\\
 u,v &\mapsto uI_{rd},vI_{rd}\\
 w &\mapsto u\overline{\alpha_{u}}+v\overline{\alpha_{v}}.
\end{align*}

With the natural grading on $N=\bigoplus_{rd}K[u,v]$, the above
morphism is a graded homomorphism. This makes $N$ into a graded
$S_{K}$-module. So $\widetilde{N}$ is a (coherent) sheaf on
$X=\mbox{Proj}\ S_{K}=C_{K}$. Note that $(q_K)_{*}\widetilde{N}\cong \bigoplus_{rd}\mathcal{O}_{\mathbb{P}^{1}}$. (See the Introduction.) But we can prove more:
\begin{lem}\label{const}
(Lemma 4.3, \cite{rK03}.) $\widetilde{N}$ is a rank $r$, degree $r(d+g-1)$ vector bundle on $X$. This bundle is stable, and it has $H^{0}(\widetilde{N}(-1))=0$.
\end{lem}
\begin{proof}
For the first statement, by Proposition 2.5.1, Vol. 4 Part 2, and Lemma
12.3.1, Vol. 4 Part 3, \cite{aG65}; it is enough to prove the
statement for $\widetilde{N}$ with the assumption that $K$ is
algebraically closed. We will prove that for any closed point $x \in
X$, $\mbox{dim}_{K}(\widetilde{N} \otimes_{\mathcal{O}_{X,x}} K)=r$.
By the usual upper semicontinuity argument, this is sufficient.
Furthermore, it is obvious that $u$ and $v$ cannot be both in a
homogeneous maximal ideal of $S_{K}$. So it is enough to prove the
dimension condition above for any closed point $x$ in
$X_{v}=Spec\ (S_{K})_{(v)}$, because the argument is the same for
$X_{u}$. Now we have:
\[
 (S_{K})_{(v)}=\frac{K[\overline{u},\overline{w}]}{(\overline{w}^{d}-f(\overline{u},1))},
\]
and
\[
 N_{(v)}=\bigoplus_{rd}K[\overline{u}].
\]
Here, $\overline{u}=u/v$ and $\overline{w}=w/v$. $\overline{u}$ acts
in a natural way and $\overline{w}$ acts as
$\overline{u}\overline{\alpha_{u}}+\overline{\alpha_{v}}$. Since $K$ is
algebraically closed, any
closed point in $Spec\ (S_{K})_{(v)}$ can be written as
$\mathfrak{m}=(\overline{u}-a,\overline{w}-b)$ for some $a,b \in K$.
So we have:
\[
 \mathcal{O}_{Spec\ (S_{K})_{(v)},x} \cong
 (\frac{K[\overline{u},\overline{w}]}{(\overline{w}^{d}-f(\overline{u},1))})_{(\overline{u}-a,\overline{w}-b)}
\]
and:
\[
 (N_{(v)})_{x} \cong \mathcal{O}_{x} \otimes_{(S_{K})_{(v)}}
 (\bigoplus_{rd}K[\overline{u}]).
\]
Using these, we get:
\[
 (\widetilde{N})_{x} \otimes_{\mathcal{O}_{X,x}} K \cong
 \frac{\bigoplus_{rd}
 K[\overline{u}]}{(\overline{u}-a,\overline{w}-b)(\bigoplus_{rd}K[\overline{u}])}\\
 \cong
 \frac{\bigoplus_{rd}K}{(a\overline{\alpha_{u}}+\overline{\alpha_{v}}-b)(\bigoplus_{rd}K)}
\]
So the required dimension is $dim_{K}(ker(a\overline{\alpha_{u}}+\overline{\alpha_{v}}-b))$. We have to prove that this dimension is equal to $r$. Let us assume that $f(a,1) \neq 0$ at first. In this case, there are exactly $d$ points $(a,b)$  such that $b^{d}=f(a,1)$. Over each of these points, the rank of the stalk $(\widetilde{N})_{x} \otimes_{\mathcal{O}_{X,x}}K$ is at least $r$ by upper  semicontinuity. Since they must add up to $rd$, each of them must be equal to $r$.

Consider $a\overline{\alpha_{u}}+\overline{\alpha_{v}} \in M_{rd}(K)$. We compute the dimension of its eigenspace of eigenvalue $b$. The characteristic polynomial of $a\overline{\alpha_{u}}+\overline{\alpha_{v}}$ is $t^{rd}-f(a,1)$. This follows from the fact that when $f(a,1) \neq 0$ all the roots are distinct (remember our initial assumption that $char(K)$ does not divide $d$); and if $f(a,1)=0$, then the matrix is nilpotent. Also, if $b \neq 0$, then $f(a,1) \neq 0$ and $b$ is an eigenvalue of multiplicity 1.

Next, let $b=0$. Then we have $f(a,1)=0$. We can find a matrix $B \in GL_{rd}(K)$ such that $B(a\overline{\alpha_{u}}+\overline{\alpha_{v}})B^{-1}$ is in Jordan form. If $dim_{K}(ker(a\overline{\alpha_{u}}+\overline{\alpha_{v}}-b)) > 1$, then we can write, for $w \in M_{rd}(K[u,v])$, $det\ w=det\ BwB^{-1}=(av-u)^{l}det\ w^{'}$, where:
\[
 BwB^{-1}=
\left(
 \begin{array}{ccccccccc}
  av-u &  . & . & . & .  & . \\
  . &  . & . & . & .  & . \\
  . &  . & av-u & . & .  & . \\
  . &  . & . & 1 & .  & . \\
  . &  . & . & . & .  & . \\
  . &  . & . & . & .  & 1
 \end{array}
\right)w^{'}.
\]
There are $l$ diagonal entries $(av-u)$ in the above matrix, and $l
\geq r+1$. But then, we have $det\ w^{d}=(av-u)^{ld}det(w^{'})^{d}=f(u,v)^{d}$. But we immediately see
that $(av-u)$ is a repeated factor of $f(u,v)$ with multiplicity at
least $l \geq 2$, which is a contradiction. So the required
dimension condition is proved.

For the second part, consider the projection $q_{K}:C_{K} \to
\mathbb{P}^{1}_{K}$. Then we have
$\chi(C_{K},\widetilde{N})=\chi(\mathbb{P}^{1}_{K},
(q_{K})_{*}\widetilde{N})$, and by Riemann-Roch:
\[
 r(1-g)+deg(\widetilde{N})=rd(1-0)+deg((q_{K})_{*}\widetilde{N}).
\]
But $(q_{K})_{*}(\widetilde{N}) \cong \bigoplus_{rd}
\mathcal{O}_{\mathbb{P}^{1}_{K}}$ and so its degree is 0. This gives
us $\mbox{deg }\widetilde{N}=r(d+g-1)$.

For the statement that $h^{0}(C_{K},\widetilde{N}(-1))=0$, note that
we can use the projection formula to get:
\[
 h^{0}(C_{K},\widetilde{N}(-1))=h^{0}(\mathbb{P}^{1}_{K},(q_{K})_{*}\widetilde{N}(-1))=0.
\]

Lastly, we have to prove that the vector bundle $\widetilde{N}$ as constructed above is stable.
For this, we follow the discussion in \cite{mvdB87}. As in the introduction, we have a vector bundle
$\widetilde{N}$ on $C$ such that $(q_K)_{*} \widetilde{N} \cong \mathcal{O}_{\mathbb{P}^{1}_{k}}^{rd}$.
We will make use of the following formula:

\begin{align}\label{formula}
 deg((q_K)_{*} \widetilde{N})/rk((q_K)_{*} \widetilde{N}) &=
 \frac{deg(\widetilde{N})-rk(\widetilde{N})(d+g-1)}{d rk(\widetilde{N})} \\
 &= \frac{1}{d}\frac{deg(\widetilde{N})}{rk(\widetilde{N})}+\frac{1-g}{d}-1
\end{align}

Suppose $\widetilde{N}$ is strictly semistable. Then let $\mathcal{F} \subseteq \widetilde{N}$ be a subbundle
such that $deg(\mathcal{F})/rk(\mathcal{F})=deg(\widetilde{N})/rk(\widetilde{N})$.
It follows from a formula similar to \ref{formula} that $deg((q_K)_{*}\mathcal{F})=0$. Recall that since $(q_K)_{*}\mathcal{F}$ is a subsheaf of the torsion-free sheaf $(q_K)_{*}\widetilde{N}\cong \mathcal{O}_{\mathbb{P}^{1}_{k}}^{rd}$, it is torsion-free itself and hence is a vector bundle. See Lemma 5.2.1, \cite{jL97}. By Lemma 4.4.1, \cite{jL97}, it is a sum of line bundles on the projective line. Hence, $(q_K)_{*}\mathcal{F}\cong \bigoplus_{i=1}^{t}\mathcal{O}_{\mathbb{P}^{1}}(n_i)$, where $\sum n_i = 0$. Now note that $H^0((q_K)_{*}\widetilde{N})=0$ and hence $H^0((q_K)_{*}\mathcal{F})=0$ as well. Hence, $n_i \leq 0$ for all $i$, and since $\sum n_i=0$, we have $n_i=0$ for all $i$. So we have $(q_K)_{*}\mathcal{F} \cong \mathcal{O}_{\mathbb{P}^{1}}^t$. It follows that the corresponding representation is reducible, which is contrary to our assumptions. This finishes the proof.
\end{proof}
Next, we prove a relative version of Lemma \ref{const}. Let $S$ be a $k$-scheme, and let $(\psi,\mathcal{O})$ be an element of $\mathcal{G}_{rd}(C_{f},S)$ so that $\mathcal{O}=\mathcal{E}nd_{\mathcal{O}_{S}}(\mathcal{E}^{S})$ for some vector bundle $\mathcal{E}^{S}$ of rank $rd$ on $S$. We will first construct a rank $r$ vector bundle on $C \times_{k} S$. So consider the graded sheaf homomorphism
\begin{align*}
 \frac{\mathcal{O}_{S}[u,v,w]}{w^{d}-f(u,v)} &\to
 \mathcal{E}nd(\mathcal{E}^{S})[u,v] \\
 u,v &\mapsto u,v \\
 w &\mapsto u\psi(x)+v\psi(y)
\end{align*}

where $x$ and $y$ are the standard generators of $C_{f}$. This is a
graded homomorphism of sheaves of $\mathcal{O}_{S}[u,v]$-algebras because the degree of
$u$ and $v$ is $1$ on the right side. We can view the right hand
side of the above morphism as
$\mathcal{E}nd_{\mathcal{O}_{S}}(\mathcal{E}^{S} \otimes_{\mathcal{O}_{S}}\mathcal{O}_{S}[u,v])$.
So it allows us to view
$\mathcal{E}^{S} \otimes_{\mathcal{O}_{S}}\mathcal{O}_{S}[u,v]$ as a
sheaf of graded $(\mathcal{O}_{S}[u,v,w])/(w^{d}-f(u,v)))$-modules.
Since $C \times_{k} S \cong \mbox{Proj}
(\mathcal{O}_{S}[u,v,w])/(w^{d}-f(u,v))$, we get a sheaf
$\mathcal{M}$ over $C \times_{k} S$. For the rest of the section,
for any point $s$ in $S$, $q_{s}$ denotes the morphism $C \times_{k}
k(s) \to \mathbb{P}^{1}_{k(s)}$ induced by the inclusion $k(s)[u,v]
\to (k(s)[u,v,w])/(w^{d}-f(u,v))$, $p_{s}$ denotes the projection of
the second factor of $\mathbb{P}^{1}_{k(s)}$ onto $Spec\ k(s)$, and $\pi_{s}$ denotes the composition $p_{s} \circ q_{s}$. We use a similar notation for an arbitrary field $K$ instead of $k(s)$.

\begin{lem}\label{const_M}
The sheaf $\mathcal{M}$ is a rank $r$ vector bundle on $C \times_{k} S$ of fiberwise constant degree of $r(d+g-1)$. Moreover, for any closed point $s \in S$, the rank $r$ vector bundle $\mathcal{M}_{s}=\mathcal{M} \otimes_{\mathcal{O}_{S}} Spec\ k(s)$ on $C_{k(s)}$ is stable, has degree $r(d+g-1)$ and satisfies $h^{0}(C_{k(s)},\mathcal{M}_{s}(-1))=0$.
\end{lem}
\begin{proof}
It is sufficient to prove these assertions when $S$ is affine. So
let $S=Spec\ R$. Let $C_{R}=C \times_{k} S$. Assuming further
(without loss of generality) that $\mathcal{E}$ is trivial on $S$,
we have $H^{0}(S,\mathcal{E}^{S}) \cong \bigoplus_{rd}R$, because
$\mathcal{E}$ is a vector bundle of rank $rd$. In this case, let us
denote the graded $(R[u,v,w]/(w^{d}-f(u,v))$-module
$H^{0}(S,\mathcal{E}) \otimes_{R} R[u,v]$ by $M$ and the
corresponding sheaf $\mathcal{M}$ by $\widetilde{M}$.

Note that $\widetilde{M}$ is flat over $S$ and that $\pi:C \times S \to S$ is a flat
morphism. So by Lemma 12.3.1, Vol. 4 Part 3, \cite{aG65}, it will be
enough to prove that for any $s \in S$, $\widetilde{M}_{s}$ is a
stable, rank $r$ vector bundle on $S$. Let $\widetilde{M}_{s} = (1
\times i_{s})^{*} \widetilde{M}$, where $1 \times i_{s}$ is the
morphism $C_{s} = C \times_{k} k(s) \to C \times_{k} S$. Consider
the representation associated to the point $s$ via $\psi$:
\[
 \psi_{s}:C_{f} \to M_{rd}(k(s)).
\]
Using Lemma \ref{const}, we obtain a
sheaf $\widetilde{N}$ which is isomorphic to $\widetilde{M}_{s}$. So
by Lemma \ref{const} we know that $\widetilde{M}_{s}$ is a rank $r$,
stable vector bundle. Also, along the fibers of the projection onto
the second factor, the degree is $r(d+g-1)$.

The second assertion follows from Lemma \ref{const}.
\end{proof}

Let $S$ be a $k$-scheme and $\mathcal{V}$ be a vector bundle on $S$. Recall that $\mathcal{V}[u,v]^{\sim}$ defines a coherent sheaf on $\mathbb{P}^{1}_{S}$. It can be shown that $p_S^*(\mathcal{V})$ is isomorphic to $\mathcal{V}[u,v]^{\sim}$. (See, for example, Exercise III-47, \cite{EH00}.) We prove the following lemma which will be used in the main theorem.
\begin{lem}\label{lem-important}
We have $\Gamma_* (\mathcal{V}[u,v]^{\sim}) \cong \mathcal{V}[u,v]$.
\end{lem}
\begin{proof}
Using the definition of $\Gamma_* (\mathcal{V}[u,v]^{\sim})$, we get
{\allowdisplaybreaks
\begin{align*}
 \Gamma_*(\mathcal{V}[u,v]^{\sim}) &= \bigoplus_{n \in \mathbb{Z}} (p_S)_* ((\mathcal{V}[u,v]^{\sim})(n)) \\
 &\cong \bigoplus_{n \in \mathbb{Z}} (p_S)_* (p_S^* (\mathcal{V})(n)) \\
 &\cong \bigoplus_{n \in \mathbb{Z}} \mathcal{V} \otimes (p_S)_*(\mathcal{O}_{\mathbb{P}^{1}_{S}}(n)) \\
 &\cong \mathcal{V} \otimes \Gamma_*(\mathcal{O}_{\mathbb{P}^{1}_{S}}) \\
 &\cong \mathcal{V} \otimes \mathcal{O}_{S}[u,v] \\
 &\cong \mathcal{V}[u,v].
\end{align*}}
\end{proof}
We are now in a position to prove the main theorem of this section:
\begin{thm}\label{Main_Theorem}
\textbf{(Main Theorem)} Any given degree $rd$ irreducible $S$-representation $(\psi,\mathcal{O}_{A})$ can be obtained as the pull-back of the representation $(\psi,\mathcal{E}nd(\pi_{*}\mathcal{E})^{PGL(N)})$ by a unique map $f:S \to U$. In particular, $(\psi,\mathcal{A})$ represents the functor $\mathcal{R}ep_{rd}(C_f,-)$.
\end{thm}
\begin{proof}
It is already known (see Theorem 4.1, \cite{mvdB89}) that the functor $\mathcal{R}ep_{rd}(C_{f},-)$ is representable. Let $T$ denote the scheme that represents it and let $(\Psi,\mathcal{B})$ be the universal representation. Since we have a $U$-representation $(\psi,\mathcal{A})$, we obtain a unique map $\alpha: U \to T$ such that $(\psi,\mathcal{A}) \cong \alpha^{*} (\Psi,\mathcal{B})$. We will construct another map $\beta: T \to U$ and prove that $\alpha$ and $\beta$ are inverses of each other.

Let $(\phi, \mathcal{E}nd(\mathcal{E}^{S}))$ be a degree $rd$ $S$-representation in $\mathcal{G}_{rd}(C_{f},-)$. Using Lemma \ref{const_M}, we can construct a vector bundle $\mathcal{M}$ on $C \times S$ such that for any point $s \in S$, $\mathcal{M}_s$ is stable, has rank $r$ and degree $r(d+g-1)$ and $H^0(\mathcal{M}_s(-1))=0$. Recall that $\mathcal{M}(r,r(d+g-1))$ is a coarse moduli space of vector bundles and $U \subseteq \mathcal{M}(r,r(d+g-1))$. So using the coarse moduli property, we obtain a map $f: S \to \mathcal{M}(r,r(d+g-1))$ whose image lies in $U$.

Since it will be necessary to use the vector bundle $\mathcal{E}$ on $C \times \Omega$ later in this proof, we use the local lifts of $f$ to $\Omega$. Using the local universal property of $\Omega$, we see that $S$ can be covered by Zariski open sets $S_{i}$ such that the restriction $f_{i}$ of $f$ to $S_{i}$ can be lifted (not uniquely) to $\Omega$. In other words, we get maps $g_{i}: S_{i} \to \Omega$ with $f_{i}=\mathfrak{q} \circ g_{i}$. (Recall that $\mathfrak{q}:\Omega \to U$ is the good quotient map.) These maps satisfy $(id \times g_{i})^{*} \mathcal{E} \cong \mathcal{M}_{i}$ on $S_{i}$, where $\mathcal{M}_{i}$ is the restriction of $\mathcal{M}$ to $C \times S_i$:
\[
\begin{CD}
S_i		 @= 	S_i \\
@Vg_{i}VV      	@VVf_{i}V \\
\Omega   @>\mathfrak{q}>>	U
\end{CD}
\]
We claim that $g_{i}^{*} (\chi,\mathcal{E}nd(\pi_{*}\mathcal{E}))$ is equivalent to (see Definition \ref{s-rep}) $(\phi,\mathcal{E}nd(\mathcal{E}^{S_i}))$, where $\mathcal{E}^{S_{i}}$ is the restriction of $\mathcal{E}^S$ to $C \times S_i$. To prove this claim, we will make extensive use of the following diagram of maps:
\[
\begin{CD}
C \times S_{i} 	@>id_{C}\times g_{i}>>	C \times \Omega \\
@Vq_{S_{i}}VV														@VVqV \\
\mathbb{P}^{1} \times S_{i}	@>{id_{\mathbb{P}^{1}}\times g_{i}}>> \mathbb{P}^{1} \times \Omega \\
@Vp_{S_{i}}VV													@VVpV \\
S_{i}	@>g_{i}>> \Omega
\end{CD}\
\]

On top of this diagram, we have the vector bundle $\mathcal{M}_{i}$ constructed over $C \times S_i$ as in Lemma \ref{const_M} and the bundle $\mathcal{E}$ on $C \times \Omega$ as in Grothendieck's theorem \ref{groth}. In Lemma \ref{const_M}, $\mathcal{M}_i$ was defined to be $(\mathcal{E}^{S_i} \otimes_{\mathcal{O}_{S_i}} \mathcal{O}_{S_i}[u,v])^{\sim}$, which was viewed as a graded $(\mathcal{O}_{S_i}[u,v,w])/(w^{d}-f(u,v))$-module. We can see from this that $(q_{S_i})_{*}\mathcal{M}_i \cong \mathcal{E}^{S_i}[u,v]^{\sim}$, viewed as a graded $\mathcal{O}_{S_i}[u,v]$-module. Next we compute $\Gamma_{*}(\mathcal{M}_{i})$. Recall that this is a graded $(\mathcal{O}_{S_i}[u,v,w])/(w^{d}-f(u,v))$-module. Then we have
{\allowdisplaybreaks
\begin{align*}
 \Gamma_{*}(\mathcal{M}_{i}) &= \bigoplus_{n \in \mathbb{Z}} (p_{S_i})_* (q_{S_i})_* (\mathcal{M}_{i}(n)) \\
                             &\cong \bigoplus_{n \in \mathbb{Z}} (p_{S_i})_* ((q_{S_i})_*(\mathcal{M}_{i}) \otimes_{\mathcal{O}_{\mathbb{P}^{1}_{S_i}}} \mathcal{O}_{\mathbb{P}^{1}_{S_i}}(n)) \\
                             &\cong \Gamma_{*}((q_{S_i})_*\mathcal{M}_{i}) \\
                             &\cong \Gamma_{*}(\mathcal{E}^{S_i}[u,v]^{\sim}) \\
                             &\cong \mathcal{E}^{S_i}[u,v].
\end{align*}
}
Here, the last line follows from Lemma \ref{lem-important}. Similarly, we compute $\Gamma_{*} ((id_C \times g_i)^* \mathcal{E})$. Recall also that this is a graded $(\mathcal{O}_{S_i}[u,v,w])/(w^{d}-f(u,v))$-module.
{\allowdisplaybreaks
\begin{align*}
 \Gamma_{*} ((id_C \times g_i)^* \mathcal{E}) &= \bigoplus_{n \in \mathbb{Z}} \pi_{*} (id_C \times g_i)^* \mathcal{E}(n) \\
       &\cong \bigoplus_{n \in \mathbb{Z}} (p_{S_i})_* (q_{S_i})_*(((id_C \times g_i)^* \mathcal{E}) \otimes_{\mathcal{O}_{C \times S_i}} \mathcal{O}_{C \times S_i}(n)) \\
       &\cong \bigoplus_{n \in \mathbb{Z}} (p_{S_i})_* ((q_{S_i})_* (id_C \times g_i)^* \mathcal{E} \otimes_{\mathcal{O}_{\mathbb{P}^{1}_{S_i}}} \mathcal{O}_{\mathbb{P}^{1}_{S_i}}(n)) \\
       &\cong \bigoplus_{n \in \mathbb{Z}} (p_{S_i})_* ((id_{\mathbb{P}^{1}_{S_i}}\times g_i)^* (q_{\Omega})_* \mathcal{E}\otimes_{\mathcal{O}_{\mathbb{P}^{1}_{S_i}}} \mathcal{O}_{\mathbb{P}^{1}_{S_i}}(n)) \\
       &\cong \bigoplus_{n \in \mathbb{Z}} (p_{S_i})_* ((id_{\mathbb{P}^{1}_{S_i}}\times g_i)^* (p_{\Omega})^* (\pi_* \mathcal{E})\otimes_{\mathcal{O}_{\mathbb{P}^{1}_{S_i}}} \mathcal{O}_{\mathbb{P}^{1}_{S_i}} (n) )\\
       &\cong \bigoplus_{n \in \mathbb{Z}} (p_{S_i})_* ((p_{S_i})^* g_i^* (\pi_* \mathcal{E})\otimes_{\mathcal{O}_{\mathbb{P}^{1}_{S_i}}} \mathcal{O}_{\mathbb{P}^{1}_{S_i}}(n)) \\
       &\cong g_i^*(\pi_* \mathcal{E}) \otimes \bigoplus_{n \in \mathbb{Z}} (p_{S_i})_* (\mathcal{O}_{\mathbb{P}^{1}_{S_i}}(n)) \\
       &\cong g_i^*(\pi_* \mathcal{E}) \otimes \Gamma_*(\mathcal{O}_{\mathbb{P}^{1}_{S_i}}) \\
       &\cong g_i^*(\pi_* \mathcal{E}) \otimes \mathcal{O}_{S_i}[u,v] \\
       &\cong g_i^*(\pi_* \mathcal{E})[u,v].
\end{align*}
}
Since $\mathcal{M}_i$ and $(id_C \times g_i)^* \mathcal{E}$ are isomorphic as sheaves, it follows that $\Gamma_{*}(\mathcal{M}_{i})$ and $\Gamma_{*} ((id_C \times g_i)^* \mathcal{E})$ are isomorphic as $(\mathcal{O}_{S_i}[u,v,w])/(w^{d}-f(u,v))$-modules.

Recall that the representation $(\phi,\mathcal{E}nd(\mathcal{E}^{S_i}))$ is defined by the action of $w$ on $\mathcal{E}^{S_i}[u,v]$ and the representation $g_i^* (\chi,\mathcal{E}nd(\pi_*\mathcal{E}))$ is similarly defined by the action of $w$ on $g_i^*(\pi_* \mathcal{E})[u,v]$. Now let $\eta: \Gamma_{*}(\mathcal{M}_{i}) \cong \mathcal{E}^{S_i}[u,v] \to \Gamma_{*} ((id_C \times g_i)^* \mathcal{E}) \cong g_i^*\pi_* \mathcal{E}[u,v]$ be an isomorphism of $(\mathcal{O}_{S_i}[u,v,w])/(w^{d}-f(u,v))$-modules. Let $w$ act on $\mathcal{E}^{S_i}$ by $u \phi^{1}_{u} + v \phi^{1}_{v}$ and on $g_i^* \pi_*\mathcal{E}$ by $u \phi^{2}_{u} + v \phi^{2}_{v}$. So we have
\[
 \eta(w) = w \eta
\]
that is,
\[
 \eta((u \phi^{1}_{u} + v \phi^{1}_{v})) = (u \phi^{2}_{u} + v \phi^{2}_{v})(\eta)
\]
Comparing the $u$ and $v$ components, we see that $\eta \phi^{1}_{u} \eta^{-1} = \phi^{2}_{u}$ and $\eta \phi^{1}_{v} \eta^{-1} = \phi^{2}_{v}$. Hence the two representations $(\phi,\mathcal{E}nd(\mathcal{E}^{S_i}))$ and $g_i^*(\chi,\mathcal{E}nd(\pi_*\mathcal{E}))$ are equivalent.

Now we have
\begin{align}\label{eq2}
 (\alpha \circ f_{i})^{*}(\Psi,\mathcal{B}) &\cong f_{i}^{*} (\psi,\mathcal{A})\\
 					   &\cong g_{i}^{*} \mathfrak{q}^{*} (\psi,\mathcal{A}) \\
                       &\cong g_{i}^{*} (\chi,\mathcal{E}nd(\pi_{*}\mathcal{E})) \\
                       &\cong (\phi,\mathcal{E}nd(\mathcal{E}^{S_i})).
\end{align}

Next, define morphisms of functors as follows:
\[
 \mathcal{G}_{rd}(C_{f},S) \to Hom_{Sch/k}(S,U) \to Hom_{Sch/k}(S,T)
\]
by sending $(\phi,\mathcal{E}nd(\mathcal{E}^{S}))$ to the map $f$ defined as above first, and then by sending $f$ to the composition $\alpha \circ f$. Hence the second map is induced by $\alpha: U \to T$. Recall that we are using the flat topology on $(Sch/S)$. Since the sheafification of $\mathcal{G}_{rd}(C_{f},-)$ is $\mathcal{R}ep_{rd}(C_{f},-)$, and $Hom_{Sch/k}(-,T)$ is a sheaf; we get morphisms of sheaves as follows:
\begin{align}\label{eq1}
 \mathcal{R}ep_{rd}(C_{f},-) \cong  Hom_{Sch/k}(-,T) \to Hom_{Sch/k}(-,U) \to Hom_{Sch/k}(-,T).
\end{align}
This sequence induces a sequence of morphisms of schemes $T \to U \to T$; and any $(\phi,\mathcal{E}nd(\mathcal{E}^{S}))$ is mapped to $\alpha \circ f$ by this composition. Recall that the second map is $\alpha$, and we will call the first map $\beta$.

Next, let $S$ be a $k$-scheme and let $(\gamma,\mathcal{G})$ be an $rd$-dimensional irreducible $S$-representation of $C_f$. Since $\mathcal{G}$ is Azumaya, by Proposition \ref{flat} we can cover $S$ by flat maps $S_i \to S$ such that the pullback of $(\gamma,\mathcal{G})$ is of the form $(\phi,\mathcal{E}nd(\mathcal{E}^{S_i}))$ for some vector bundle $\mathcal{E}^{S_i}$ on $S_i$. This gives us a map $f_i:S_i \to U$ as stated earlier in the proof. Then we can cover $S_i$ by Zariski open subsets $S_{i,j}$ over which we can lift the restrictions $f_{i,j}:S_{i,j}  \to U$ to $\Omega$. Following the formula \ref{eq2}, we see that the pullback of the universal representation $(\Psi,\mathcal{B})$ along $\alpha \circ f_{i,j}$ is isomorphic to $(\phi_{i,j},\mathcal{E}nd(\mathcal{E}^{S_{i,j}}))$. This means that given any section of the sheaf $\mathcal{R}ep_{rd}(C_{f},-)$ corresponding to an irreducible $rd$-dimensional representation $(\gamma,\mathcal{G})$ over a $k$-scheme $S$, there is a flat cover of $S$ by maps $S_{i,j} \to S$ such that the morphism of functors defined in the formula \ref{eq1} maps the restriction of $(\gamma,\mathcal{G})$ to $S_{i,j}$ to the unique map $S_{i,j} \to T$ that pulls the universal representation $(\Psi,\mathcal{B})$ back to $(\gamma,\mathcal{G})$. Since $Hom_{Sck/k}(-,T)$ is a sheaf, this proves that the composition in \ref{eq1} is the identity. In other words, $\alpha \circ \beta$ is the identity.

Since $\alpha \circ \beta$ is the identity on $T$, and since $(\psi,\mathcal{A})$ pulls back to $(\Psi,\mathcal{B})$ under $\beta$, it is clear that $(\psi,\mathcal{A})$ must pull back to $(\Psi,\mathcal{B})$ under $\beta$.

We claim that $\beta$ is the inverse of $\alpha$. We only need to prove that the composition $\beta \circ \alpha: U \to T \to U$ is the identity on $U$. It is clear that the $U$-representation $(\psi,\mathcal{A})$ pulls back to itself under $\beta \circ \alpha$.

Denote $\delta=\beta \circ \alpha:U \to U$ for brevity. We claim that $\delta$ maps a closed point $x \in U$ to itself. To see this, consider the composition $\delta \circ i_{x}:Spec\ k(x) \to U \to U$. Then $i_{x}^{*}(\psi,\mathcal{A})$ is the irreducible $rd$-dimensional representation of $C_f$ corresponding to the closed point $x$ by Proposition \ref{important_remark}. Let $y=\delta(x)$, it is clear that the inclusion morphism of $y$ is $\delta \circ i_{x}$. We have
\begin{align*}
 (\delta \circ i_x)^{*}(\psi,\mathcal{A}) &\cong i_x^{*} \delta^{*}(\psi,\mathcal{A}) \\
 &\cong i_x^{*}(\psi,\mathcal{A}).
\end{align*}
Again, by Proposition \ref{important_remark}, $y$ must be the closed point corresponding to the $rd$-dimensional representation $i_x^{*}(\psi,\mathcal{A})$. But this is the point $x$. It is now clear that the map $\delta$ is the identity on closed points. Since $U$ is a variety, it now follows from Lemma \ref{easy-lemma} that $\beta \circ \alpha$ is the identity map on $U$. This finishes the proof.
\end{proof}

\section{Galois descent}
In the last section, we proved Theorem \ref{Main_Theorem} under the assumption that the base field $k$ is algebraically closed. In this section, we assume that the binary form $f(u,v)$ is defined over a perfect, infinite field $k$ whose characteristic does not divide $d$; and we prove the main theorem in this case.

We denote the algebraic closure of $k$ by $k'$. Then $k'/k$ is Galois. Let $G=Gal(k'/k)$ denote the Galois group. Throughout this section, we denote the constructs with base field $k'$ with a superscript $'$.

Let $V'$ be a variety over $k'$. Recall that a \emph{model} for $V'$ is a variety $V$ over $k$ such that $V' \cong V \times_{k} k'$.

\begin{prop}\label{p:gal}
There is a model $U$ for $U'$, where $U'$ is the moduli space as constructed in Section \ref{s:con} over $k'$.
\end{prop}
\begin{proof}
Recall (Th\'{e}or\`{e}me 3.1, \cite{aG61}) that the Quot scheme $Q$ parametrizing quotients of $\mathbb{E}$ as defined in \ref{mb} is defined for an arbitrary base. We have an open subscheme $\Omega'$ in $Q'$. Since the canonical map $Q' \to Q$ is open, the image $\Omega$ of $\Omega'$ is open in $Q$. It is now clear that there is a $PGL(N)$-action on $\Omega$ that induces the $PGL(N)$-action on $\Omega'$ and hence a uniform geometric quotient $U$ of $\Omega$ by $PGL(N)$ by Proposition 1.9, Chapter 1, \cite{dM82}. Since this quotient is uniform, we have $U'=U \times k'$.
\end{proof}

We now describe how to obtain a $U$-representation $(\psi,\mathcal{A})$ that is a model for the universal representation. We consider the trivial bundle $\mathbb{E}$ over $k$ with rank $N$ as described in \ref{mb}, and we consider the Quot scheme $Q$ over $k$ parametrizing the quotients of $\mathbb{E}$ that have rank $r$ and degree $r(d+g-1)$. (Recall that prescribing the rank and degree is equivalent to prescribing the Hilbert polynomial of a vector bundle by the Riemann-Roch theorem.) Consider the open subset $\Omega$ of $Q$ as described in Proposition \ref{p:gal}. There is a universal quotient bundle $\mathcal{E}$ over $\Omega$ as described in Section \ref{mb}. To show that $(\psi',\mathcal{A}')$ descends to $k$, we mimic the construction in Section \ref{s:con}. The proofs of Lemma \ref{pf} and Theorem \ref{usls} carry over. We can then follow the construction of the universal representation as in Theorem \ref{psi} to obtain a sheaf of Azumaya algebras $\mathcal{A}$ and a $k$-algebra homomorphism $\psi:C_f \to H^0(\mathcal{A})$.

We now show that $(\psi,\mathcal{A}) \times k'$ is isomorphic to $(\psi',\mathcal{A}')$. By Proposition 9.3, Chapter 3, \cite{rH77}, we have $\pi'_{*}(\mathcal{E}') \cong \pi_{*}(\mathcal{E}) \times k'$, where $\pi': C \times \Omega' \to \Omega'$ or $\pi:C \times \Omega \to \Omega$ denotes the projection. It is now clear that $\mathcal{E}nd(\pi'_{*}\mathcal{E}')\cong \mathcal{E}nd(\pi_{*}\mathcal{E}) \times k'$. Since $\mathcal{A}'$ is defined to be $\mathcal{E}nd(\pi'_{*}\mathcal{E}')^{PGL(N)}$ and $\mathcal{A}$ is defined to be $\mathcal{E}nd(\pi_{*}\mathcal{E})^{PGL(N)}$; $\mathcal{A}'$ descends to $\mathcal{A}$.

It now remains to show that the two sections $\alpha_u'$ and $\alpha_v'$ of $H^0(\mathcal{A}')$ defined in Theorem \ref{psi} descend to two sections $\alpha_u$ and $\alpha_v$ of $H^0(\mathcal{A})$. To see this, recall that $\alpha_u'$ and $\alpha_v'$ were defined by the action of $w$ (considered as an element of $\mathcal{O}_{\Omega'}[u,v,w]/(w^d-f)$) on $\Gamma_{*}(\mathcal{F}')$. Here, $\mathcal{F}'$ denotes the pushforward $(q_{\Omega'})_{*}(\mathcal{E}')$. Now, since the pushforward commutes with flat base change, we have $(q_{\Omega'})_{*}(\mathcal{E}') \cong (q_{\Omega})_{*}(\mathcal{E})\times k'$. This implies that $\Gamma_{*}(\mathcal{F}')\cong\Gamma_{*}(\mathcal{F})'$. Consider the two sections $\alpha_u$ and $\alpha_v$ of $H^0(\mathcal{A})$ obtained by the action of $w$ on $\Gamma_{*}(\mathcal{F})$. Since $\Gamma_{*}(\mathcal{F}')\cong\Gamma_{*}(\mathcal{F})'$, and since the action of $w$ is compatible with base change, it follows that $\alpha_u'$ and $\alpha_v'$ descend to $\alpha_u$ and $\alpha_v$ respectively. This proves that $(\psi,\mathcal{A}) \times k'$ is isomorphic to $(\psi',\mathcal{A}')$.

We now prove the main result of this section.
\begin{thm}
$U$ represents $\mathcal{R}ep_{rd}$.
\end{thm}
\begin{proof}
Let $S$ be a $k$-scheme. Denote $S'=S \times_{k} k'$. We have a map $\mathcal{R}ep_{rd}(S) \to \mathcal{R}ep'_{rd}(S')$ by taking an $S$-representation $(\psi,\mathcal{O}_{A})$ to the $S'$-representation $(\psi',\mathcal{O}_{A}\times_{k} k')$ where $\psi'$ is the induced $k'$ algebra homomorphism $C_f \to H^0(\mathcal{O}_{A}\times_{k}k')$. It is clear that $(\psi',\mathcal{O}_{A}\times_{k} k')$ is invariant under the action of $G$. Hence we obtain a map $\mathcal{R}ep_{rd}(S) \to \mathcal{R}ep'_{rd}(S')^{G} \cong Hom_{Sch/k'}(S',U')^{G}=Hom_{Sch/k}(S,U)$. Here, $G$ acts on $Hom_{Sch/k'}(S',U')$ by conjugation. In order to show that this map is a bijection, we need to show that the first arrow is a bijection. Injectivity follows from Theorem 6, pg. 135, \cite{BLR90}. To prove surjectivity, consider an element $(\psi',\mathcal{O}_{A}')$ in $\mathcal{R}ep'_{rd}(S')^{G}$. This means that, for every $\sigma \in G$ considered as an automorphism of $S'$, $\sigma^{*}(\psi',\mathcal{O}_{A}')$ is equivalent to $(\psi',\mathcal{O}_{A}')$. In other words, we can find an isomorphism $\phi_{\sigma}: \sigma^{*}\mathcal{O}_{A}' \to \mathcal{O}_{A}'$ such that the map $\psi':C_f \to H^0(\mathcal{O}_{A}')$ is equal to the composition $H^0(\phi_{\sigma})  \circ \sigma^{*}\psi'$. Here, $\sigma^{*}\psi'$ is the induced map $C_f \to H^0(\sigma^{*}\mathcal{O}_{A}')$. Now consider $\phi_{\tau} \circ \tau \phi_{\sigma} \circ \phi_{\tau \sigma}^{-1}$. This is an automorphism of $\mathcal{O}_{A}'$ that is the identity on the two global sections given by the images of $x_1$ and $x_2$ in $C_f$. Since these two global sections generate $\mathcal{O}_{A}'$ locally, $\phi_{\tau} \circ \tau \phi_{\sigma} \circ \phi_{\tau \sigma}^{-1}$ must be the identity. Hence the $\phi_{\sigma}$ form a descent datum for $(\psi',\mathcal{O}_{A}')$ and there is an $S$-representation $(\psi,\mathcal{O}_{A})$ that is a model for $(\psi',\mathcal{O}_{A}')$. This proves the surjectivity of the map as above and finishes the proof.
\end{proof}
% ----------------------------------------------------------------

\end{document}